\documentclass[12pt]{amsart} \usepackage{amscd}
\usepackage[all]{xy}
\usepackage{amssymb}
\UseComputerModernTips

\newtheorem{theorem}{Theorem}
\newtheorem{lemma}[theorem]{Lemma}
\newtheorem{corollary}{Corollary}
\newtheorem{proposition}[theorem]{Proposition}

\theoremstyle{definition}                 
            \newtheorem{defn}{Definition}
\newtheorem{remark}{Remark} \newtheorem*{notation}{Notation}

\usepackage{amsfonts}
\newcommand{\field}[1]{\mathbb{#1}}          \newcommand{\Q}{\field{Q}}
\newcommand{\R}{\field{R}}                   \newcommand{\Z}{\field{Z}}
\newcommand{\C}{\field{C}}

 \newcommand{\ra}{\rightarrow}

\begin{document}

\title[Image of The Burau Representation at $d$-th Roots of Unity]
{Image of The Burau Representation at $d$-th Roots of Unity}

\author{T.N.Venkataramana}

{\address{ T.N.Venkataramana, School of Mathematics, TIFR, Homi Bhabha
Road, Colaba, Mumbai 400005, India}

\email{venky@math.tifr.res.in}

\subjclass{primary:  22E40.  Secondary:  20F36  \\  T.N.Venkataramana,
School of  Mathematics, Tata  Institute of Fundamental  Research, Homi
Bhabha Road, Colaba, Mumbai 40005, INDIA}

\date{}

\begin{abstract} We show  that the image of the  braid group under the
monodromy action on the homology of a cyclic covering of degree $d$ of
the  projective line  is an  arithmetic group  provided the  number of
branch points is  sufficiently large compared to the  degree $d$. This
is  deduced by proving  the arithmeticity  of the  image of  the braid
group  on $n+1$ letters  under the  Burau representation  evaluated at
$d$-th roots of unity when $n\geq 2d$.
\end{abstract}

\maketitle{}

\tableofcontents

\newpage

\section{Introduction}\label{introsection}

This  paper is  concerned with  the question  of whether  some natural
monodromy   groups  are  arithmetic.    These  questions   were  first
considered  by   Griffiths  and  Schmid   \cite  {Gr-Sch}.   Following
\cite{Gr-Sch}  pp. 123-124,  we say  that a  subgroup  $\Gamma \subset
GL_N(\Z)$ is an  {\it arithmetic group} if $\Gamma  $ has finite index
in the integral points of its  Zariski closure in $GL_N$. If $\Gamma $
is an arithmetic  group then by a result  of Borel and Harish-Chandra,
$\Gamma $  is a  lattice in the  group of  real points of  the Zariski
closure.  In \cite{Gr-Sch}, the groups $\Gamma $ arise
as follows.  Let  $\pi : X \ra S$ be an  algebraic family of algebraic
manifolds, which  is a fibration.   Write $V_{s_0}=\pi ^{-1}(\{s_0\})$
for a typical  fibre and $\Gamma \subset Aut  (H^m(V_{s_0}))$ the {\it
monodromy group}  i.e. the image of  the action of $\pi  _1(S)$ on the
cohomology   group   (with  integral   coefficients)   of  the   fibre
$V_{s_0}$.  Griffiths  and Schmid  then  raise  the  question: is  the
monodromy group an arithmetic group?  \\

The foregoing question has a  negative answer in general, as was shown
by  Deligne  and  Mostow   (\cite{Del-Mos})  (there  are  examples  of
monodromy groups which are not even finitely presented in \cite{Nor} -
arithmetic groups are finitely  presented by a result of Raghunathan).
Let us recall  the basic set-up of \cite{Del-Mos}.  Let  $d \geq 2$ be
an integer. For the family $S$, take the space of complex $n+1$-tuples
$a=  (a_1, \cdots,  a_{n+1})$  whose entries  are  all distinct.   Fix
integers $k_1, \cdots, k_{n+1}$ with $1\leq k_i \leq d-1$. Given $a\in
S$, consider the set $X_{a,k}$ of solutions ($x,y$) of the equation
\[y^d=  (x-a_1)^{k_1}(x-a_2)^{k_2}\cdots  (x-a_{n+1})^{k_{n+1}}.\] The
space $X_{a,k}$ has  a natural structure of a  compact Riemann surface
$X_{a,k}^*$ with finitely many punctures. We then get a map $\pi: X\ra
S$ described on the ``affine part'' by  $(x,y,a) \mapsto a$, where
$(x,y)\in X_{a,k}$  and $a\in S$  and where $X$  is the family  of the
compact  Riemann  surfaces  $X_{a,k}^*$;  we  can  then  consider  the
monodromy action on $H^1(X_{a,k}^*,\Z)$ for a typical fibre $X_{a,k}^*$. \\

The fundamental  group of the space $S$  is well known to  be the pure
braid group $P_{n+1}$ on $n+1$  strands; thus the fibration $X \ra S$
yields a monodromy representation
\[\rho _M ^* (k,d): P_{n+1}\ra GL(H^1(X_{a,k} ^* ,\Z)),\] of $P_{n+1}$
on the integral cohomology of the fibre $X_{a,k}^*$.  If $N$ is the rank
of the abelian group $H_1(X_{a,k}^*,\Z)$,  then the image $\Gamma $ of
$P_{n+1}$  is  a subgroup  of  $GL_N(\Z)$.  Form  the Zariski  closure
${\mathcal  G}$ of $\Gamma  $ in  $GL_N$; this  is a  linear algebraic
group defined over $\Q$. \\  

We  now  give  only  a  qualitative  description  of  the  results  of
\cite{Del-Mos}. Deligne  and Mostow prove that for  special choices of
the  integers $d,n$ and  $k_1, \cdots,  k_{n+1}$, the  monodromy group
$\Gamma $  does not  have finite index  in ${\mathcal  G}(\Z)$.  Hence
this  gives a  negative answer  to the  question  of Griffiths-Schmidt
mentioned before.  It  can be shown that the  group ${\mathcal G}(\R)$
of  real  points  of  the  Zariski  closure is  a  product  $\prod  _j
U(p_j,q_j)$ of  unitary groups $U(p_j,q_j)$ (the  unitary structure on
$H^1(C, \C)$  of the curve  $C=X_{a,k}^*$ comes from  the intersection
form $h(\alpha,\beta )= \int _C \alpha \wedge {\overline \beta}$).  As
we  mentioned before,  ${\mathcal G}(\Z)$  is a  discrete  subgroup of
${\mathcal G}(\R)$ and hence  so is $\Gamma$.  However, the projection
of  $\Gamma  $  to  one  of  these factors  $U(p_j,q_j)$  may  not  be
discrete. In  \cite{Del-Mos} (see also \cite{Mos}) it  is shown -using
their  INT and  $\Sigma$-INT criteria-  that  for a  finite number  of
special choices of  $k_i, d,n$, one of these  $U(p_j,q_j)$ is $U(n,1)$
and that the {\it projection} of the monodromy group $\Gamma $ to this
factor  $U(n,1)$ continues  to be  discrete.  Once  the  projection is
discrete, it follows (see \cite{Mc}, Theorem (10.3), for example) that
the image of the projection is  a lattice in $U(n,1)$; if there is one
more non-compact factor isomorphic  to $U(p,q)$ in ${\mathcal G}(\R)$,
then it  follows (see \cite{Mc}, \S 10,  p. 48) that the  image of the
projection  of the  $\Gamma $  in  $U(n,1)$ is  a non-arithmetic  {\it
lattice}.  In particular, the monodromy  group $\Gamma $ does not have
finite  index in  ${\mathcal G}(\Z)$  (if the  monodromy were  to have
finite  index,  then  the  projection  to  $U(n,1)$  would  either  be
non-discrete  or an  arithmetic  lattice; this  is  discussed in  more
detail in \cite{Mc}, Corollary 10.4). \\

In  view  of  the  Margulis arithmeticity  theorem  (that  irreducible
lattices in  linear semi-simple Lie groups  of real rank  at least two
are arithmetic), the above strategy to produce non-arithmetic lattices
cannot work if $U(n,1)$ is replaced  by $U(p,q)$ with $p,q \geq 2$; if
we are to have  $U(n,1)$ as a factor, it may be  shown that the number
$n+1$ of  branch points  $a_1, \cdots, a_{n+1}$  must not  exceed $2d$
(see  the proof of  Theorem \ref{mainth}  where we  prove this  in the
special case when all the $k_i$  are $1$).  However, if we take $n\geq
2d$,  it   is  still  of  interest -in  view  of   the  question  of
Griffiths and Schmid- to  know  whether the  monodromy $\Gamma  \subset
{\mathcal  G}(\Z)\subset  GL_N(\Z)$  has  finite index  in  ${\mathcal
G}(\Z)$.  In this paper, we prove that if $n\geq 2d$ and all the $k_i$
are $1$, then $\Gamma $  does have finite index in ${\mathcal G}(\Z)$
(see Theorem \ref{cyclicmonodromy} for the statement of a more general
case). \\

Consider  the  compactification  $X_a^*$  of the  affine  curve  $X_a$
defined by
\[y^d= (x-a_1)(x-a_2)\cdots (x-a_{n+1}), \]  with $y\neq 0$ and $x\neq
a_1, \cdots,  a_{n+1}$.  As before,  there is the monodromy  action of
the pure braid group $P_{n+1}$ on the cohomology of the fibre $X_a^*$.
Since the equation of the curve is patently invariant under the action
of the permutations of the $a_i$'s, the action extends to an action of
the full braid group $B_{n+1}$ on $H_1(X_a^*,\Z)$.

\begin{theorem}  \label{fullbraidmonodromy}  Suppose  $d\geq  3$.   If
$n\geq 2d$, then  the image $\Gamma $ of  the monodromy representation
$\rho  (d): B_{n+1}\ra  GL(H^1(X_a^*,\Z))=GL_N(\Z)$  is  an  arithmetic
group. \\

Moreover, the monodromy  is a product of  irreducible lattices, each
of  which is  a non-co-compact  arithmetic group  and has  $\Q$-rank at
least two.
\end{theorem}

\begin{remark}  The group $G=\Z/d\Z$,  viewed as  the group  of $d$-th
roots of unity, operates on each of the curves $X_a^*$ for varying $a$
and hence acts on the first cohomology $H^1(X_a^*,\Z)$; therefore, one
may view  the first cohomology group  of $X_a^*$ as a  module over the
group algebra $\Z[G]=\Z[q]/(q^d-1)$.  The  action of $G$ commutes with
the monodromy action of the  braid group, and therefore, the monodromy
group lies in the space  of endomorphisms of $H^1(X_a^*,\Z)$ which are
$\Z[G]$  module  maps. Moreover,  the  monodromy  group preserves  the
intersection  form $(\alpha,  \beta )\mapsto  \alpha \wedge  \beta$ on
$H^1(X_a^*,\Z)$ which extends as a Hermitian form on $V=H^1(X_a^*,\C)$
given by  $(\alpha,\beta )\mapsto  \alpha \wedge {\overline  \beta }$.
Therefore, the monodromy  lies in the unitary group  of this Hermitian
form,  and preserves  the eigenspaces  $V_{\eta}$ (and  hence  the sum
$W_{\eta}=  V_{\eta }\oplus  V_{\overline  \eta}$) of  the group  $G$.
Consequently,  $W_{\eta}$  is  a  Hermitian space  and  the  monodromy
representation  restricted to $W_{\eta}$  has its  image in  a unitary
group of the form $U(r,s)$  with $r,s$ depending on $\eta$.  Since the
cohomology  group  $H^1(X_a^*,\C)$  is  a  direct sum  of  the  spaces
$W_{\eta}$, it follows that the image of the monodromy group lies in a
product of the $U(r,s)$. \\

The image of $\Gamma $ in the individual $U(r,s)$ may not be discrete, but 
$\Gamma $ as a subgroup of the product will be discrete, 
since it preserves the integral lattice $H^1(X_a^*,\Z)$.  
\end{remark}

\begin{remark} The arithmetic group  is specified (up to finite index)
in      Proposition       \ref{arithmeticspecify}      in      section
\ref{monodromyandburau}. It  is a little complicated  to describe when
$n+1$ and $d$ are not coprime, but Proposition \ref{arithmeticspecify}
shows that  the group  of real  points of the  Zariski closure  of the
monodromy is a product of unitary groups $U(r,s)$.
\end{remark}

\begin{remark} If $n+1\leq d$ then the group of integral points of the
Zariski  closure  of  the   monodromy  is  a  product  of  irreducible
arithmetic lattices (the integral points of the Zariski closure are as
described in Proposition  \ref{arithmeticspecify}), some of which form
{\it co-compact} lattices of their real Zariski closures.  The Zariski
closure can  in fact  be shown to  be a  product of unitary  groups of
suitable Hermitian forms over  certain totally real number fields, and
one  of  these  Hermitian  forms  is  definite  at  one  of  the  real
completions of the  number field. We give a  more detailed analysis of
the Hermitian form later. \\

A    result    of     A'Campo    \cite{A'C}    says    that    Theorem
\ref{fullbraidmonodromy}  holds  even when  $d=2$.   In  the proof  of
Theorem  \ref{fullbraidmonodromy},  we use  the  fact  that a  certain
central element in $B_{n+1}$ acts by a {\it non-trivial} scalar on the
Burau  representation  if   $n$  is  of  the  form   $kd-2$  in  Lemma
\ref{scalar}; it is here that we need that $d\geq 3$. The proof can be
slightly  modified  to  extend  to  the case  $d=2$  (see  the  remark
following Lemma \ref{scalar}), but we will not do so here.  \\

Theorem   \ref{fullbraidmonodromy}  answers   a  question   raised  in
\cite{Mc} (see  question 11.1 in \cite{Mc}) in  the affirmative when
$n\geq 2d$. When $n\leq d-2$, the monodromy is, in some cases, {\it not
arithmetic} as  is shown by  the examples of Deligne-Mostow  (cf. the
example of $d=18$ and $n=3$ of Corollary (11.8) of \cite{Mc}). \\

In a different direction (\cite{Al-Car-Tol}), the arithmeticity of the
image of  the braid group  of type $E_6$  into $U(4,1)$ is  proved; in
\cite{Al-Hal} representations  of the braid  group on the  homology of
{\it non-cyclic}  coverings are considered.  Arithmeticity results for
cyclic coverings of compact surfaces are proved in \cite{Looi2}.
\end{remark}

The monodromy  representation of the braid  group $B_{n+1}$ considered
in Theorem \ref{fullbraidmonodromy} is  closely related to the reduced
Burau   representation   of   the   group   $B_{n+1}$   (see   section
\ref{monodromyandburau} for details). Theorem \ref{fullbraidmonodromy}
is  deduced  from  the  arithmeticity  of  the  images  of  the  Burau
representation of the  braid group $B_{n+1}$ at $d$-th  roots of unity
(analogously,    a    more     general    result,    namely    Theorem
\ref{cyclicmonodromy},       is        deduced - in       section
\ref{cyclicmonodromysection}- from the  arithmeticity of the images of
the Gassner representation of the pure braid group $P_{n+1}$ evaluated
at  $d$-th  roots of  unity).  In the  rest  of  the introduction,  we
concentrate  only  on  the  Burau   case;  the  case  of  the  Gassner
representation is much  more involved and we postpone  dealing with it
to a future occasion.

\subsection{The braid group and the Burau representation}

\subsubsection{Definition} The  braid group $B_{n+1}$  on $(n+1)$-strands
is the free group on  $n$ generators $s_1,s_2, \cdots, s_n$ modulo the
relations
\[s_js_k=s_ks_j~(\mid   j-k\mid  \geq  2),   \quad  {\rm   and}  \quad
s_js_ks_j=s_ks_js_k ~(\mid j-k\mid =1).\]

\subsubsection{The reduced Burau representations $\rho _n$ }

Let $R=\Z [q,q^{-1}]$  be the Laurent polynomial ring  in the variable
$q$  with integral  coefficients.  Let  $M=R^n$ be  the  standard free
$R$-module  of rank  $n$  with standard  generators $e_1,e_2,  \cdots,
e_n$.   For each  $j$ define  the  operator $T_j\in  End_R(M)$ by  the
formulae
\[T_j(e_j)=-qe_j,~~T_j(e_{j-1})=e_{j-1}+qe_j,~~
T_j(e_{j+1})=e_{j+1}+e_j,\] and
\[T_j(e_k)=e_k ~~(\mid k-j\mid \geq 2).\] In the above formulae, we do
not  attach any  meaning to  $T_1(e_0)$.  Similarly  $T_n(e_{n+1})$ is
left  undefined. The  map  $s_j\mapsto T_j$  defines a  representation
$\rho _n:  B_{n+1}\ra GL_n(R)$.   Denote by $\Gamma  _n$ the  image of
$\rho _n$.  The representation $\rho  _n$ is the reduced  {\bf Burau
representation in degree $n$}  (\cite{Bir}).

\subsubsection{A Hermitian form on $R^n$}

The  ring $R=\Z[q,q^{-1}]$ has  an involution  $f\mapsto \overline{f}$
given by $\overline{f}(q)=f(q^{-1})$. The sub-ring $S$ of invariants in
$R$ under this involution is clearly $\Z[q+q^{-1}]$. \\

There is a unique map $h:R^n\times  R^n\ra R$, which is a bilinear map
of $S$-modules, so that for all $v,w\in R^n$ and all $\lambda, \mu \in
R$ we have
\[\overline{h(v,w)}=h(w,v),~~ h(\lambda v,\mu w)=\lambda \overline{\mu
} h(v,w),\] and
\[h(e_j,e_k)=0 \quad (\mid j-k \mid \geq 2),~~ \]
\[h(e_j,e_{j+1})=-(q+1),  \quad  h(e_j,e_j)=  \frac{(q+1)^2}{q}.\]  We
denote  this form by  $h=h_n$ (  and when  $n$ is  fixed, we  drop the
subscript, and  write $h$).  Then $\Gamma _n$  preserves the Hermitian
form  $h$ on  $R^n$. We  can then  talk of  the unitary  group  of the
Hermitian form $h$. \\

For a  detailed description of a  unitary group as  an algebraic group
defined over a field $K$,  see the beginning of Section \ref{unitary}.
The unitary  group $U(h)$  of the Hermitian  form $h$ is  an algebraic
group scheme defined over $S$ and
\[U(h)(S)=\{g\in GL_n(R):  h(gv,gw)=h(v,w)~~\forall v,w\in R^n  \}. \]
More generally, for any commutative  $S$ algebra $A$, $U(h)(A)$ is the
group
\[U(h)(A)= \{ g\in GL_n(R\otimes  _S A); h(gv,gw)=h(v,w) ~~\forall v,w
\in R^n \otimes _S A\}.\]

\begin{remark} This Hermitian form $h$ is essentially ( up to a scalar
multiple and  equivalence of Hermitian forms) the one  constructed by
Squier \cite{Sq}  where he  uses  a  form  with coefficients  in  a
quadratic extension $R'=\Z[\sqrt{q}, {\sqrt  q}^{-1}]$ of $R$. We have
written it in this form since  we need an algebraic group defined over
$S$ and not over $R$. \\

It is  customary to use  the notation $U(R,h)$  in place of  $U(h)$ to
denote a unitary group (see  \cite{Pl-Ra}, section 2.3.3), in order to
specify  the  quadratic extension  $R/S$  with  respect  to which  the
Hermitian form $h$ is defined. Since our quadratic extension is fixed,
we have used $U(h)$ to avoid complicating the notation.
\end{remark}

\subsection{The Burau  representations $\rho_n(d)$ at  $d$-th roots of
unity}

If  $\mathfrak{a}\subset R$ is  an ideal  stable under  the involution
$f\mapsto \overline{f}$, and  $A=R/\mathfrak{a}$ is the quotient ring,
then on  the $A$-module  $A^n$ we get  a corresponding  Hermitian form
$h_A$,  and  a corresponding  representation  $\rho _n(A):  B_{n+1}\ra
GL_n(A)$,  which maps  $B_{n+1}$  into $U(h_A)(B)$  where  $B$ is  the
quotient ring $S/\mathfrak {a}\cap S$. \\

We can take ${\mathfrak a}$ to be the principal ideal in $R$ generated
by  $\Phi _d(q)$  where  $\Phi _d(q)$  denotes  the $d$-th  cyclotomic
polynomial  in $q$.   Then the  quotient $R_d=R/{\mathfrak  a}$  is an
integral  domain  (the  ring  of  integers in  the  $d$-th  cyclotomic
extension $E_d$  of $\Q$).  The reduction modulo  the ideal $\mathfrak
a$  of the  representation  $\rho _n$  yields  a representation  $\rho
_n(d)$  of the  braid group  $B_{n+1}$.  This  is referred  to  as the
representation obtained by evaluating  the representation $\rho _n$ at
({\it all the}) primitive $d$-th roots of unity. \\

Note  that there  is no  preferred embedding  of the  cyclotomic field
field $E_d=\Q[q]/(\Phi _d(q))$ into the field $\C$ of complex numbers,
since there is no preferred choice of a primitive $d$-th root of unity
in $\C$.  We  will therefore not consider $E_d$ as  a sub-field of $\C$
but  as being  naturally embedded  in the  product $\prod  _{\mu }\C$,
where the  product is over  all the primitive  $d$-th roots $\mu  $ of
unity (the map into the product  is obtained by evaluating $q$ at {\it
all} the primitive $d$-th roots of unity). \\

Denote  by  $\Gamma  _n(d)$  the  image of  the  representation  $\rho
_n(d)$. The image goes into the group $U(h)(O_K)$ where $K=K_d$ is the
(totally  real)  sub-field of  $E_d$  invariant  under the  involution
$f\mapsto \overline{f}$: $K=\Q(2cos(\frac{2\pi}{d}))$ and $O_K$ is the
ring of  integers in  $K$.  We will  consider arithmetic  subgroups of
(i.e. subgroups which have finite  index in) $U(h)(O_K)$ (we refer to
subsection \ref{arithmeticgroups}  for a  discussion on why  these are
arithmetic groups in the sense of \cite{Gr-Sch}).  \\

Consider the  group $\Gamma  _n(d)\subset U(h)(O_K)$. The  ambient Lie
group  in which $U(h)(O_K)$  is naturally  a lattice  (see the  end of
subsection \ref{arithmeticgroups})  is the product  group $G_{\infty}=
\prod _{v \mid  \infty} U(h)(K_v)$ where the product  runs through all
the archimedean completions $K_v$ of  $K$.  Since $K$ is totally real,
$K_v$ is isomorphic to $\R$ for  each $v$.  The form $h$, however, may
be  different for  different  real  embeddings of  $K$.   when $h$  is
non-degenerate, there exist non-negative integers $r_v$ and $s_v$ with
$r_v+s_v=n$ such  that the unitary group $U(h)(K_v)$  is isomorphic to
$U(r_v,s_v)$ as an algebraic group  over $\R$ (when $h$ is degenerate,
the unitary group $U(h)(K_v)$ has a unipotent radical and the quotient
by the unipotent radical is of the form $U(r_v,s_v)$ with $r_v+s_v\leq
n-1$). In  the course of the  proof of arithmeticity  of monodromy, we
never need  to use  the ambient group  $G_{\infty}$; we  work directly
with  the arithmetic group  $G(O_K)$ and  the monodromy  group $\Gamma
\subset G(O_K)$.   For these reasons,  we do not specify  the integers
$r_v$, $s_v$ and the ambient Lie group $G_{\infty}$.

\subsubsection{Statement of results}

The main result of the paper is  

\begin{theorem}\label{mainth} If  $d\geq 3$  and $n\geq 2d$,  then the
image $\Gamma  _n(d)$ of the  Braid group $B_{n+1}$ under  the reduced
Burau representation  $\rho _n$ evaluated  at {\rm all  the} primitive
$d$-th roots of  unity {\rm [} namely the image
of  representation   $\rho  _n(d):  B_{n+1}\ra  GL_n(\Z[q,q^{-1}]/(\Phi
_d(q)))$ {\rm ]} is an arithmetic group. \\

More  precisely,  if  $h$  is   the  Hermitian form  on  $A^n$  which
$\Gamma_n(d)$ preserves,  then $\Gamma _n(d)$ is a  subgroup of finite
index in $U(h)(O_K)$,  where $K=\Q(cos(\frac{2\pi}{d}))$ is the totally
real sub-field of the $d$-th cyclotomic extension of $\Q$. \\

These  arithmetic  groups  are  of  $\Q$-rank at  least  two,  and  in
particular, are not co-compact lattices.
\end{theorem}

We now take $d=3,4,6$. In these cases, the $d$-th cyclotomic extension
$E_d=\Q(e^{2\pi i/d})$ is an imaginary quadratic extension of $\Q$ and
the totally real sub-field $K_d=\Q({\rm cos}(\frac{2\pi }{d}))$ is the
field  $\Q$ of  rationals, $O_d=\Z$  and $\Gamma  \subset U(h_n)(\Z)$.
Combining  Theorem  \ref{mainth} with  the  results of  Deligne-Mostow
(\cite{Del-Mos}, \cite{Mc}) we will prove:

\begin{theorem}\label{imagquad} If $d=3,4,6$ then {\bf for all n}, the
image of the Braid group  $B_{n+1}$ under the representation $\rho _n$
evaluated  at  a primitive  $d$-th  root  of  unity is  an  arithmetic
subgroup. \\

More precisely, the image of $\rho _n(d)$ is an arithmetic subgroup of
the integral unitary group  $U(h)(\Z)$.
\end{theorem}

We refer to Section  \ref{imagquadproofsection} for a proof of Theorem
\ref{imagquad}. \\

Now consider the ring $A=\Z[q,q^{-1}]/(q^d-1)$, where $R=\Z[q,q^{-1}]$
as before  . The free $A$  module $A^n$ of  rank $n$ may be  viewed in
particular, as a free Abelian group of rank $nd$, and $GL_n(A)$ can be
viewed as a subgroup of  $GL_{nd}(\Z)$. We say that a subgroup $\Gamma
\subset GL_n(A)$ is arithmetic, if it has finite index in its integral
Zariski closure in $GL_{nd}(\Z)$.  A Theorem of \cite{A'C} and Theorem
\ref{mainth} together imply the following:

\begin{theorem} \label{arithmetic} Consider the Burau representation
\[\rho: B_{n+1}\ra GL_n(\Z[q,q^{-1}]/(q^d-1)).\] Then the
image of  $\rho$ is an arithmetic  group for all $d\geq  2$ and $n\geq
2d$.
\end{theorem}

\subsection{Description of the proof}

The proof of  Theorem \ref{mainth} is by showing  that for $n\geq 2d$,
the image $\Gamma _n(d)$ contains a large number of unipotent elements
(precisely,  $\Gamma _n(d)$  contains  an arithmetic  subgroup of  the
unipotent radical of a parabolic $\Q$- subgroup).  By using results of
Bass-Milnor-Serre  and Tits  (\cite{Ba-Mi-Se},  \cite{Ti}), and  their
extensions  to  other groups  (\cite{Ra},  \cite{Va},  \cite {Ve})  on
unipotent generators  for noncocompact arithmetic  groups of $\R$-rank
at  least two,  we  show (in  section  \ref{proofssection}) that  such
groups are arithmetic if $n\geq 2d$. \\

The  proof that  $\Gamma _n(d)$  contains sufficiently  many unipotent
elements (see  section \ref{proofssection}) is by  using induction. We
first prove this in the case when  $n\geq 2d$ and $n$ is a multiple of
$d$. Then we prove an inductive step that if we can get unipotents for
$m$, then we  get unipotents for $m+1$ ($m \geq  2d$).  This will then
cover all integers $n\geq 2d$. \\

The construction of sufficiently many unipotent elements is especially
easy to  describe when the representation is  the Burau representation
of $B_{n+1}$ at $d$-th roots  of unity and $n=2d$ (or, more generally,
when  $n=kd$ is  a multiple  of $d$,  with $k\geq  2$;  see subsection
\ref{nequals2dcase}).    We   will   exploit   the   fact   that   the
representation $\rho  _{2d-1}(d)$ is  not irreducible, but  contains a
nonzero  invariant  vector  (see Proposition  \ref{burauirred}).   Let
$s_1$ be as before. Denote by $\Delta '$ the product element
\[\Delta   '    =   (s_2s_3\cdots   s_n)(s_2s_3\cdots   s_{n-1})\cdots
(s_2s_3)(s_2).\]  It can  be shown  that  $\Delta '^2$  is central  in
$B_n'$, the  braid group generated by  $s_2,s_3,\cdots,s_n$.  Form the
commutator  $u=[s_1,(\Delta  ')^2]$.   Consider  the  group  $U\subset
B_{n+1}$ generated  by conjugates of $u$ of  the form $\{huh^{-1}:h\in
B_n'\}$. We show (in subsection \ref{nequals2dcase}) that the image of
this group $U$ under the Burau representation at $d$-th rots of unity,
is  an arithmetic  subgroup of  the unipotent  radical of  a (maximal)
parabolic subgroup of the unitary group $U(h)$. \\

This is enough to prove that the image of $B_{n+1}$ under $\rho _n(d)$
is arithmetic, by the criteria of section \ref{alggroups}. \\

In     section    \ref{monodromyandburau},    we     derive    Theorem
\ref{fullbraidmonodromy}  from Theorem  \ref{mainth},  by establishing
the  precise  relationship  between  the monodromy  representation  of
Theorem   \ref{fullbraidmonodromy}  and   the   Burau  representation;
although this  connection is  essentially well known  (cf.  \cite{Mc},
Theorem  5.5),  we  will  give  a  more  precise  description  of  the
connection in section \ref{monodromyandburau}.

\subsection{Organisation of the paper}

This paper  is organised as  follows.  In Section  \ref{alggroups}, we
recall some basic notions from  algebraic groups and state a criterion
due to Bass-Milnor-Serre and others on unipotent generators for higher
rank non-uniform  lattices.  In  Section \ref{unitary}, we  will apply
this criterion to certain integral unitary groups. \\

The  main section of  the paper  is Section  \ref{burausection}.  In
Section \ref{burausection}, we show that the image of the braid
group  $B_{n+1}$ at  a primitive  $d$-th root  of unity  contains many
unipotent elements (more precisely, contains an arithmetic subgroup of
the unipotent radical  of a parabolic subgroup defined  over the field
$K_d$).   The  criteria  of   section  \ref{unitary}  then  imply  the
arithmeticity when $n\geq 2d$ for all $d$. \\

In  Section  \ref{applications},  we  first  consider  (in  subsection
\ref{mcmullen}),  complex  reflection  groups  corresponding  to  root
systems of  type $A$ and show  the arithmeticity of the  images of the
corresponding  Artin groups  $A_n(q)$ (see  Section \ref{applications}
for references to definitions) where $q$ is a primitive $d$-th root of
unity and $n\geq 2d$; the image  of the group $A_n(q)$ turns out to be
the same  as the image  of $\Gamma _n(d)$  into one of the  factors of
$U(h)(K\otimes \R)$.  Hence the arithmeticity of the  image of $A_n(q)$
will be shown to be  an immediate consequence of Theorem \ref{mainth}.
This answers question 5.6 in \cite{Mc}, in many cases.  \\

In  the   next  subsection   (\ref{sarnak}),  we  show   that  Theorem
\ref{mainth} implies the arithmeticity of the monodromy of certain one
variable hypergeometric functions of type $_nF_{n-1}$, in some special
cases of the  parameters. The point here is  that arithmeticity can be
proved for  an {\it infinite  family} of parameters associated  to the
hypergeometric equations $_nF_{n-1}$. \\

\noindent {\bf  Acknowledgements:} I am  very grateful to  Madhav Nori
for mentioning to me this problem of arithmeticity of the image of the
braid group  ( Theorem  \ref{mainth});  his  formulation of  the
problem in purely algebraic terms  was most helpful.  I also thank him
for generously sharing his insights on many of the questions addressed
here  and  for  very  helpful  remarks  on  many  of  the  proofs  and
intermediate results (especially the  remark that an earlier proof for
the arithmeticity of the image of the Burau representation should also
go through for the Gassner representation). \\

I  thank  Peter  Sarnak   for  mentioning  the  monodromy  problem  of
subsection  \ref{sarnak} and  for very  interesting discusions  on the
material of subsection \ref{sarnak}.\\

I thank Curtis Mcmullen for helpful communications concerning question
(5.6) of \cite{Mc}, and pointing  out some corrections on the material
in  subsection \ref{mcmullen}.  \\

I extend to the referee my  hearty thanks for a careful reading of the
MS  and for  pointing  out numerous  corrections  and suggestions  for
improving  the  exposition of  the  paper. \\

The  support of the  J.C.Bose  fellowship for  the period  2008-2013 is
gratefully acknowledged.

\section{Algebraic groups}\label{alggroups}

For the facts on algebraic groups  stated in this section, we refer to
\cite{Bor-Ti}.

\subsection{Examples of algebraic groups} 

If $K$  is a number field and  $G\subset GL_n$ is a  subgroup which is
the set of zeroes of a collection of polynomials in the matrix entries
of $GL_n$,  such that  the coefficients of  these polynomials  lie in
$K$, then $G$ is said to be an algebraic group defined over $K$. \\

For  example, the  groups  $SL_n$ and  $Sp_{2g}$  are algebraic  groups
defined over $K$.  If $E/K$  is a quadratic extension and $H:E^n\times
E^n\ra E$  is a $K$ bilinear  form which is Hermitian with respect to
the non-trivial automorphism  of $E/K$, then the unitary  group of the
Hermitian form $h$ is an algebraic group defined over $K$.

\subsection{Arithmetic groups} \label{arithmeticgroups} 

Let  ${\mathcal G}\subset GL_N$  be a  linear algebraic  group defined
over $\Q$.  A subgroup $\Gamma \subset {\mathcal G}(\Z)$ is said to be
an {\it  arithmetic subgroup of}  ${\mathcal G}(\Q)$ if $\Gamma  $ has
finite  index in the  intersection ${\mathcal  G}(\Z)={\mathcal G}\cap
GL_N(\Z)$  (see  \cite{Ra2},  Chapter  X,  def  (10.12),  p.   165  or
\cite{Pl-Ra}, Chapter 4,  (4.1), p.  171).  By a  theorem of Borel and
Harish-Chandra $\Gamma $ is a lattice (a discrete subgroup with finite
covolume) in the group ${\mathcal G}(\R)$ of real points of ${\mathcal
G}$  provided the  identity component  ${\mathcal G}^0$  of  the group
${\mathcal  G}$  does  not  have non-trivial  homomorphisms  into  the
multiplicative  group  ${\mathbb  G}_m$   defined  over  $\Q$  (for  a
reference,  see  \cite{Pl-Ra} Theorem  4.13,  p  213).  It is  also  a
consequence of the result  of Borel and Harish-Chandra that ${\mathcal
G}(\Z)$ is a  co-compact lattice if and only  if $\Q$-rank (${\mathcal
G})=0$ (see subsection \ref{rank} for the notion of $\Q$-rank).  \\

Let $O_K$ denote the ring of integers in $K$; The group $G(O_K)$ is by
definition,  the intersection $GL_n(O_K)\cap  G$.  A  subgroup $\Gamma
\subset G(K)$  is said  to be arithmetic  if the  intersection $\Gamma
\cap G(O_K)$  has finite index  in $\Gamma $ and  in $G(O_K)$.  This
definition appears to be different from the one given in the preceding
paragraph.  But, these two  definitions are equivalent.  This is shown
by replacing  $G$ by  the Weil restriction  of scalars  ${\mathcal G}=
R_{K/\Q}G$ (for  a reference, see \cite{Pl-Ra}, 2.1.2,  page 49).  The
theory of Weil restriction of  scalars says that there exists a linear
algebraic group  ${\mathcal G}$  defined over $\Q$  (and unique  up to
isomorphism)  with   the  following  property:   for  any  commutative
$\Q$-algebra $A$, the group  ${\mathcal G}(A)$ is naturally isomorphic
to  $G(K\otimes _{\Q}  A)$,  the group  of $K\otimes  _{\Q}A$-rational
points of the $K$-algebraic group $G$.

\begin{defn} The  group ${\mathcal G}$ is called  the Weil restriction
of  scalars  from  $K$ to  $\Q$,  of  the  group  $G$. It  is  denoted
${\mathcal G}=R_{K/\Q}(G)$.
\end{defn}

We then  have $R_{K/\Q}G(\Z)\simeq G(O_K)$ where the  symbol $\simeq $
means that there is equality up  to subgroups of finite index (the two
groups  are {\it  commensurable}).  Moreover,  ${\mathcal G}(\R)\simeq
\prod  _{v\mid  \infty}G(K_v)$, where  the  product  is  over all  the
inequivalent archimedean completions of $K$  (see 2.1.2, on pp. 50 and
51 of  \cite{Pl-Ra}). Thus, $G(O_K)$ is  a lattice in  the ambient Lie
group $\prod  G(K_v)$ where the  product is over all  the inequivalent
non-archimedean completions  $K_v$ of $K$.  We recall  that any number
field has $r_1$ completions $K_v$  which are isomorphic to $\R$ and up
to complex  conjugation, $r_2$ completions $K_v$  which are isomorphic
to $\C$ such  that $r_1+2r_2$ is the degree of  the extension $K$ over
$\Q$. \\

We will have  occasion to deal with images  of arithmetic groups under
morphisms     of     $\Q$-algebraic     groups     (see     subsection
\ref{prooffullbraid}).   In  particular,  we  will use  the  following
result.   Let $f: {\mathcal  G} \ra  {\mathcal G'}$  be a  morphism of
algebraic groups defined over $\Q$.  This induces a homomorphism, also
denoted $f$, from ${\mathcal  G}(\Q)$ into ${\mathcal G'}(\Q)$.  It is
elementary  that if $\Gamma  \subset {\mathcal  G}(\Z)$ is  a suitable
subgroup (a  suitable ``congruence  subgroup'') of finite  index, then
$f(\Gamma )$ lies in ${\mathcal G'}(\Z)$. For a proof of the following
lemma, see Corollary (10.14) of \cite{Ra2}.

\begin{lemma} \label{arithmeticimage} The image of  $\Gamma $ under $f$
is an arithmetic subgroup of ${\mathcal G'}(\Z)$.
\end{lemma}

Let $\theta : V\ra V'$ be a linear map of $\Q$-vector spaces and $\rho
:\Delta \ra GL(V)$  be a representation of a  group $\Delta$. Then the
composite $\theta \circ  \rho $ is a representation  of $\Delta $. The
following is immediate from Lemma \ref{arithmeticimage}.

\begin{proposition}  \label{arithmeticimageproposition}  If  the image  of
$\rho $  is an arithmetic subgroup  of $GL(V)$, then the  image of the
composite $\theta \circ \rho $ is an arithmetic subgroup of $GL(V')$.
\end{proposition}

\subsection{K-rank}  \label{rank} In  this subsection,  we  define the
notion of the $K$-rank and $\Q$-rank of a linear algebraic group.

\begin{defn} An  algebraic group  $G$ is said  to be  $K$-isotropic if
there exists an  injective morphism ${\mathbb G}_m ^r\ra  G$ of linear
algebraic groups  defined over $K$  for some $r\geq 1$.   The $r$-fold
product ${\mathbb  G}_m^r$ is called the $K$-split  torus of dimension
$r$ and the embedding is  called a $K$-embedding.  The $K$-rank of $G$
is by definition the {\it maximum},  call it $r$, of the dimensions of
$K$-split tori which are $K$-embedded  in $G$.  Let $T$ be a $K$-split
torus  of dimension $r$  which is  $K$-embedded in  $G$.  Then  $T$ is
called  a maximal  $K$-split  torus.   It is  known  that all  maximal
$K$-split tori are conjugate under the group $G(K)$.
\end{defn}

For example, for  any $K$, one can compute the  $K$-rank of the groups
$SL_n$ and $Sp_{2g}$: the $K$-rank of $SL_n$ is $n-1$; the $K$-rank of
$Sp_{2g}$  is  $g$. The  $K$-rank  may  depend on  $K$:  if  $D$ is  a
quaternionic  division  algebra  over  $\Q$, let  $G=SL_2(D)$  be  the
algebraic  group  defined  over   $\Q$.  Then  $\Q$-rank  $(G)=1$;  if
$K\subset D$  is a quadratic extension  of $\Q$, then  as an algebraic
group over $K$, $G$ is isomorphic to $SL_4$ and $K$-rank $(G)= 3$. \\

It follows from  definitions that $K$-rank $(G)$=$\Q$-rank $({\mathcal
G})$ where ${\mathcal G}$ is  the Weil restriction of scalars from $K$
to $\Q$.

\subsection{Parabolic subgroups} \label{opposite} 

Suppose now that $G\subset GL_n$  is an algebraic group defined over a
number field $K$;  embed $K$ in $\C$ and suppose  that the Lie algebra
of $G(\C)$ is a  simple Lie algebra over $\C$. Then $G$  is said to be
{\it absolutely almost simple}. \\

An algebraic subgroup $P\subset G$  is said to be a parabolic subgroup
defined over $K$, if $P$ is defined over $K$ and the quotient $G/P$ is
a projective variety. For  example a subgroup $P\subset GL_n=GL(V)$ is
a parabolic  subgroup if and only  if it is the  subgroup preserving a
partial flag
\[\{0\}\subset W_1 \subset W_2\subset \cdots  \subset W_{r-1}\subset
V,\] where $W_i$ form a  sequence of subspaces of $V$ with $W_i\subset
W_{i+1}$. \\

Let $G$  be an absolutely  almost simple algebraic group  defined over
$K$.  The group $G$ has positive  $K$-rank if and only if $G$ contains
a  proper  parabolic  subgroup  $P$  defined over  $K$.  Suppose  that
$r=K$-rank $(G)\geq 1$  and $T$ a maximal $K$-split  torus in $G$.  Then
there   exists   a   parabolic   $K$-subgroup  $P$   containing   $T$.
Furthermore,  there  exists  a  non-trivial maximal  unipotent  normal
subgroup  $U$ of $P$,  called the  unipotent radical  of $P$.  The Lie
algebra of $U$ is stable under  the action of the group $T$ and splits
into character spaces for the adjoint action of $T$.  \\

The structure theory  of parabolic subgroups says that  there exists a
parabolic  subgroup $P^{-}$ defined  over $K$  containing $T$,  with a
unipotent  radical $U^{-}$  such that  the characters  of $T$  on $Lie
(U^{-})$ are  the inverses of the  characters of $T$  on $Lie(U)$. The
group $P^{-}$ is said to be  {\it opposite} to $P$ and $U^{-}$ is said
to  be {\it  opposed} to  $U$.   The intersection  $M=P\cap P^{-}$  is
called  a Levi  subgroup of  $P$ and  we have  the  Levi decomposition
$P=MU$. The group $M$ normalises both $U$ and $U^{-}$. \\

A $K$-parabolic subgroup  $P_0$ containing $T$ is {\it  minimal} if it
is  of  the  smallest  dimension  among  the  $K$-parabolic  subgroups
containing $T$. If $P\supset P_0$ then we have the (reverse) inclusion
of  the unipotent  radicals $U\subset  U_0$.  There  exists  a minimal
parabolic subgroup $P_0^{-}$ opposed  to $P_0$, with unipotent radical
$U_0^{-}$.

\subsection{The real  rank} 

Suppose  that  $G$ is  an  absolutely  almost  simple algebraic  group
defined over a  number field $K$.  Write $G_{\infty}$  for the product
group $G_{\infty}= \prod _{v~arch}G(K_v)$ (which is a real semi-simple
group), where $v$ runs over all the archimedean completions of $K$; we
write
\[\infty {\rm -rank}(G)=  \R {\rm -rank}(G_{\infty} )=  \sum  _{v~arch}  K_v{\rm 
-rank}
(G),\] and call this the {\it real rank} of $G_{\infty}$.

\subsection{A criterion for arithmeticity}

Bass,  Milnor and  Serre proved  that for  any $N\geq  1$,  the $N$-th
powers  of  the  upper  and  lower triangular  unipotent  matrices  in
$SL_n(\Z)$  generate a  subgroup  of finite  index  in $SL_n(\Z)$  for
$n\geq  3$.   A  similar  result  holds for  $Sp_{2g}(\Z)$.   In  this
subsection, we  describe an analogous result for  all higher $\R$-rank
groups. \\

The following  result is due to  many people: for $G=  SL_n ~(n\geq 3)
~~{\rm or} \quad Sp_{2g} ~(g\geq 2)$,  this is due to Bass, Milnor and
Serre  \cite{Ba-Mi-Se}; when  $G$  is  split over  $K$  (i.e.  when  a
maximal $K$  -split torus  is also a  a maximal $\C$-split  torus), to
Tits \cite{Ti} and for classical  groups $G$ with $K$-rank $\geq 2$ to
Vaserstein  \cite{Va}.  The  case  of  a general  $G$  is  handled  in
Raghunathan \cite{Ra} and \cite{Ve}.

\begin{theorem} \label{bamise} Let $G$  be an absolutely almost simple
algebraic group defined over a  number field $K$.  Assume that $K-rank
(G)\geq  1$, $\R-rank(G_{\infty})\geq  2$, and  that  $P_0,P_0^{-}$ are
minimal parabolic  $K$-subgroups of $G$ with  unipotent radicals $U_0$
and $U_0^{-}$. Then the following hold.  \\

{\rm [1]} For every integer $N\geq 1$, let $\Delta _N$ be the subgroup
of $G(O_K)$ generated  by $N$-th powers of the  elements in $U_0(O_K)$
and  $U_0^{-}(O_K)$.    Then  $\Delta  _N$  has  finite  index  in
$G(O_K)$. \\

{\rm  [2]} If  $\Delta _N'\subset  \Delta  _N$ is  an infinite  normal
subgroup, then $\Delta _N'$ also has finite index in $G(O_K)$.
\end{theorem}

The above theorem  says that the integral points  of two {\it maximal}
opposing  unipotent  subgroups  (i.e.the  unipotent  radicals  of  two
minimal parabolic  $K$-subgroups) generate a finite  index subgroup in
$G(O_K)$  if   the  $\infty$-rank  is   at  least  two.   We   need  a
strengthening  of  this,  where  we  assume only  that  the  unipotent
radicals which are not  necessarily maximal, of two opposing parabolic
subgroups are involved. This is the following:

\begin{corollary} \label{criterion}  Suppose $G$ is  absolutely almost
simple and  $K-rank (G)\geq 2$.  Let  $P$ and $P^{-}$  be two opposite
parabolic  subgroups containing  a  maximal $K$-split  torus, and  $U,
U^{-}$ their unipotent radicals. \\

For any integer $N\geq 1$, the group $\Delta _N(P^{\pm})$ generated by
$N$-th  powers of  $U(O_K)$ and  $U^{-}(O_K)$  is of  finite index  in
$G(O_K)$.
\end{corollary}

\begin{proof}  In  the  notation  above, let  $M=P\cap  P^{-}$.   Then
$M(O_K)$ normalises $U(O_K)$ and $U^{-}(O_K)$ and hence normalises the
group $\Delta _N(P)$.  Now the group generated by $M(O_K)$ and $\Delta
_N(P^{\pm})$     contains     $(U_0\cap     M)(O_K)$    and     $(U_0\cap
U)(O_K)^N=U(O_K)^N$;    the    decomposition    $P=MU$    shows    that
$P(O_K)=M(O_K)U(O_K)$   and   hence  $U_0(O_K)=(U_0\cap   M)(O_K)
U(O_K)$ [all  these equalities are true  only up to  finite index; the
decomposition $U_0=(U_0\cap  M)U$ of  algebraic groups is  defined over
the {\it  field} $K$, but not  over the integers  $O_K$.  This implies
that  there   are  finite  index   subgroups  $U_0'\subset  U_0(O_K)$,
$(U_0\cap M)'\subset  (U_0\cap M  )(O_K)$ and $U'\subset  U(O_K)$ such
that the  product decomposition  $U_0'=(U_0\cap M)'U'$ holds  for the
smaller groups]. \\

Therefore  the group  generated by  $M(O_K)$ and  $\Delta _N(P^{\pm})$
contains $N$-th  powers of elements of  $U_0(O_K)$ and $U^{-}_0(O_K)$.
Consequently, the group $\Delta _N(P ^{\pm})$ is normalised by $\Delta
_N$.  By  the second part of  the above theorem, $\Delta  _N(P)$ is of
finite index in $G(O_K)$.
\end{proof}

\begin{remark} The second part of Theorem \ref{bamise} is true for any
irreducible lattice in a real  semi-simple group of real rank at least
two (this is the normal subgroup theorem of Margulis).
\end{remark}

In the  next section, we will  state a special case  of this corollary
for certain unitary groups.

\subsection{Subgroups of products of higher rank groups}

In a  later section (section  \ref{proofssection}), we will  prove the
arithmeticity of the image of the braid group in a group of the form
$U(h)(\Z[q]/(q^d-1))$.   The  latter  is   a  product  of  the  groups
$U(h)(O_e)$ for  certain rings  of integers $O_e$.  To deal  with this
case,  we now  prove a  lemma  which is  a simple  consequence of  the
super-rigidity theorem of Margulis. \\

\begin{lemma}  \label{products} Suppose that  $K_e$ are  number fields
for each element $e$ in a  finite indexing set $X$. Suppose that $G_e$
is  an absolutely  almost simple  semi-simple algebraic  group defined
over $K_e$, and suppose that  $\infty -rank (G_e)\geq 2$ for all $e\in
X$. Suppose  that $\Gamma \subset  \prod G_e(O_e)$ is a  subgroup such
that the image  of its projection to each  $G_e(O_e)$ has finite index
in $G_e(O_e)$.   Assume in addition,  that either $K_e$ and  $K_f$ are
non-isomorphic or else, if $K_e$  and $K_f$ are isomorphic, the groups
$G_e$ and $G_f$ {\rm (}which may be thought of as groups defined over
the same field $K_e${\rm )} are not isomorphic over $K_e=K_f$. \\
 
Under these assumptions,  the group $\Gamma $ has  finite index in the
product $\prod _{e\in X} G_e(O_e)$.
\end{lemma}
\begin{proof} Replacing the  arithmetic groups $G_e(O_e)$ by subgroups
of finite index,  we may assume that these are  torsion free and hence
that $\Gamma  $ is torsion free.   We prove the lemma  by induction on
the number  of factors.  Fix an element  $p\in X$.  By  induction, the
projection  of $\Gamma  $ in  the product  $\prod _{e\in  X,  e\neq p}
G_e(O_e)$ under the  projection map $pr: \prod _{e\in  X} G_e(O_e) \ra
\prod _{e\in X, e\neq p}  G_e(O_e)$ has finite index. Suppose $N_p$ is
kernel of  restriction of this projection  map to $  \Gamma$.  We will
show that the kernel $N_p$ cannot be trivial. \\

If  $N_p$ is  trivial, then  $\Gamma $  projects injectively  into the
product of the groups $G_e(O_e)$ with $e\neq p$: $\Gamma \subset \prod
_{e\in X,  e\neq p}  G_e(O_e)$ (and by  the induction  assumption, its
image  has  finite index).   Therefore,  $\Gamma  $  is an  arithmetic
subgroup of  the higher rank  lattice $\prod _{e\neq p}  G_e(O_e)$ and
has  a non-trivial  representation (projection  to the  $p$-th factor)
onto  a finite  index  subgroup of  the  arithmetic group  $G_p(O_p)$;
replacing the image of $\Gamma $ by a smaller subgroup of finite index
if necessary, we  may assume that the image of  $\Gamma$ in the ``away
from  $p$''  product  is  a  product  of  finite  index  subgroups  of
$G_e(O_e)$  (with $e\neq  p$).  The  Margulis normal  subgroup theorem
(applied to  the image of $\Gamma  $ in $G_e(O_e)$)  then implies that
only one of  the factors in this product  maps isomorphically onto its
image in  $G_p(O_p)$ and that the  other factors map  to the identity.
This contradicts the Margulis  super-rigidity (or Mostow rigidity): we
have an  isomorphism of a finite  index subgroup of  $G_e(O_e)$ with a
finite  index subgroup  of $G_p(O_p)$.   Such an  isomorphism,  by the
super-rigidity theorem,  is induced by  first an isomorphism  of $K_e$
with  $K_p$ and  an isomorphism  of $G_e$  with $G_p$  as  groups over
$K_e=K_p$.  This contradicts our assumptions and therefore, the kernel
$N_p$ cannot be trivial.\\

Since the kernel $N_p$ is  non-trivial, it is infinite since $\Gamma $
is assumed to be torsion free.  The conjugation action of $\Gamma $ on
the  $p$-th factor $G_p(O_p)$  factors through  its projection  to the
$p$-th  factor; but  the $p$-th  projection  map has  image of  finite
index, and hence $N_p$ is normalised  by a subgroup of finite index in
$G_p(O_p)$; by the Margulis  normal subgroup theorem, $N_p$ has finite
index  in $G_p(O_p)$;  therefore, $\Gamma  $ maps  onto a  subgroup of
finite  index in the  product of  the ``away  from $p$''  factors, and
intersects  the   $p$-th  factor  in  a  subgroup   of  finite  index.
Therefore, $\Gamma $ has finite index.
\end{proof}

\begin{remark} A related result is  proved in \cite{Gr-Lu} and also in
\cite{Looi2}. At the time of writing the present paper, we were unaware
of these papers. Moreover, the  proof here is different from the cited
papers.

\end{remark}

\section{Unitary groups}\label{unitary}

\begin{notation} Suppose  that $E/K$ is a  quadratic extension.  Write
$E=K\oplus K\sqrt{\alpha}$  for some  non-square element $\alpha  $ in
$K$ and  given $z\in E$, write  $z=x+y\sqrt{\alpha}$ accordingly. Then
$x,y$   are    called   the   ``real''    and   ``imaginary''   parts,
respectively.     Denote     by     $\overline{z}$     the     element
$x-\sqrt{\alpha}y$. \\

Given  $n\geq 1$,  the vector  space  $E^n$ may  be viewed  as a  $2n$
dimensional vector space over  $K$.  Suppose $h:E^n\times E^n\ra E$ is
a map which is $K$-bilinear such that for all $v,w\in E^n$ we have
\[h(v,w)={\overline h(w,v)}.\] Then $h$ is said to be a Hermitian form
with respect  to $E/K$.   By definition, the  elements of  the unitary
group  $U(h)$  satisfy the  property  that  they  commute with  scalar
multiplication by elements of $E$, viewed as $K$- linear endomorphisms
of $E^n$; further, they preserve  the real and imaginary parts of $h$.
These  properties characterise  the  elements of  $U(h)$  and in  this
manner, the group $U(h)$ may  be viewed as a $K$-algebraic subgroup of
$GL_K(K^{2n})=GL_{2n}(K)$. In particular,
\[U(h)(K)=\{g\in  GL_n(E)\subset GL_{2n}(K): h(gv,gw)=h(v,w)\}  \] for
all vectors $v,w  \in E^n $.  More generally, if  $A$ is a commutative
$K$ algebra, then
\[U(h)(A)=\{g\in  GL_n(E\otimes _K  A):  h(gv,gw)=h(v,w)\},\] for  all
elements $v,w$ of  the $E$- module $E^n\otimes _K  A$. \\ 

For example, if $K$  and $E$ are replaced by $\R$ and  $\C$ and $h$ is
the standard  Hermitian form on $\C^n$,  then the group  of {\it real}
points of the unitary group $U(h)$ are given by the compact group
\[U(h)(\R)=\{g\in GL_n(\C):~^t {\overline  g}g=1\},\] and the group of
{\it complex  } points of the  unitary group may easily  be seen (from
the above description) to be isomorphic to $GL_n(\C)$.

\end{notation}

\subsection{Rank  of a  unitary group}  

Let  $h_2:E^2\times E^2 \ra  E$ be  a Hermitian form with  respect to
$E/K$.  Suppose that the Hermitian form $h_2$ can be written as
\[h_2=\begin{pmatrix} 0 & 1 \\ 1 & 0\end{pmatrix}.\] That is, if $v,w$
are   column  vectors   in  $E^2$   with  entries   $v=(z_1,z_2)$  and
$w=(w_1,w_2)$,  then  $h_2(v,w)= z_1\overline{w_2}+z_2\overline{w_1}$.
Then $h$  is called  a {\it hyperbolic  form}.  The standard  basis of
$E_2$  may  be  written $v,v^*$,  and  the  form  $h_2$ is  such  that
$h(v,v)=h(v^*,v^*)=0$ and $h(v,v^*)=1$.\\

Suppose  $h$  is  a  non-degenerate  Hermitian form  on  $E^n$.   Then
$(E^n,h)$  can  be written  as  a  direct sum  of  $r$  copies of  the
hyperbolic form $(E^2,h_2)$ and a form $(E^{n-2r},h_0)$ which does not
represent  a  zero in  $E^{n-2r}$.   Then $h_0$  is  said  to be  {\it
anisotropic}. Let  $U(h)$ be the  unitary group of the  Hermitian form
$h$.   It  can be  shown  that  $K$-rank $(U(h))=r$. \\

With respect to this decomposition $h=h_2\oplus \cdots \oplus h_2\oplus h_0$,
write  the  standard  basis  of  the $j$-th  copy  of  $(E^2,h_2)$  as
$v_j,v_j^*$.  we may rearrange the basis of $E^n$ in the form
\[ v_1,v_2, \cdots,v_r, w_1,  \cdots w_s, v_r^*, \cdots v_2^*,v_1^*,\]
where $w_1,  \cdots, w_s$ is  a basis of $(E^{n-2r},h_0)$.   Let $Av_1
\subset W \subset V=E^n$, where $W$ is the $E$-submodule spanned by
\[v_1,\cdots, v_r, w_1, \cdots w_s, v_r^*, \cdots v^*_2.\]

In  the terminology  of  the  preceding section,  if  $n\geq 2$,  then
$SU(h)$ is  an absolutely almost  simple algebraic group  defined over
$K$; the subgroup $P$ of $SU(h)$ which preserves the flag $Ev_1\subset
W \subset V$ is a parabolic  subgroup defined over $K$. Let $U$ be the
subgroup of  $SU(h)$ which preserves  this flag and acts  trivially on
successive quotients. Then $U$ is the unipotent radical of $P$. \\

The partial  flag $Ev^*\subset Ev^*\oplus W \subset  V$ defines a
parabolic subgroup $P^{-}$  of $U(h)$, and $U^{-}$ is  the subgroup of
$P^{-}$ which acts trivially on the successive quotients of this flag.
This parabolic subgroup $P^{-}$ is opposite to $P$, and $U^{-}$ is its
unipotent  radical  opposed  to   $U$,  in  the  sense  of  subsection
\ref{opposite}. \\
   
\begin{corollary}  \label{unitarycorollary} Along  with  the preceding
notation  and hypotheses,  suppose  that $r \geq  2$.  Then the  group
generated  by the  $N$-th powers  of $U(O_K)$  and $U^{-}(O_K)$  is an
arithmetic subgroup of $SU(h)(O_K)$.  The same conclusion holds if the
$K$-rank of $SU(h)$ is $1$ but $\infty $ -rank $(SU(h))$ is $\geq 2$.

\end{corollary}

\begin{proof}  We have already  noted that  $K$-rank $(SU(h))$  is the
number  of   hyperbolic  $2$-planes  in  the   decomposition  of  $h$.
Therefore, the $K$-rank of $SU(h)$  is $\geq 2$.  Then the corollary
follows from Corollary \ref{criterion} of the preceding section. 

The second part follows by Theorem \ref{bamise} (the group has higher
real rank but $\Q$-rank one).
\end{proof}

\subsection{The Heisenberg  group and the  group $P$} 

Assume now  that $(V,h)$  is an $n=m+1$  dimensional $E$  vector space
with  a {\it non-degenerate}  Hermitian form  $h$.  Assume  that there
exists a $E$-vector  subspace $X$ of $V$ of  codimension two such that
$h$ is  non-degenerate on $X$  and that there exist  isotropic vectors
$v,v^*$ in  $V$ which are  orthogonal to $X$, such  that $h(v,v^*)\neq
0$.  We have a partial flag
\[0\subset Ev  \subset Ev\oplus X \subset V =Ev\oplus  X\oplus Ev^* .\]
The subgroup of  the unitary group $U(h)$ of  $V$ which preserves this
flag  and  acts  trivially  on  successive  quotients  is  called  the
Heisenberg  group $H(X)$  of $X$.  We write  $P$ for  the  subgroup of
$U(h)$ which  preserves this  flag. It  is easily seen  that $P$  is a
maximal parabolic  subgroup of the  unitary group $U(h)$  defined over
$K$.  The Heisenberg group $H(X)$  is the unipotent radical of $P$ (we
sometimes denote $H(X)$ by $U$,  to be consistent with the notation of
the  preceding section).   Since we  have an  orthogonal decomposition
$V=X \oplus (Ev\oplus  Ev^*)$ with respect to the  Hermitian form $h$,
it follows that the unitary group $M=U(h_{\mid X})$ of the restriction
of  $h$ to  $X$ is  the  subgroup of  $U(h)$ which  fixes the  vectors
$v,v^*$.  We have $P=H(X)M=UM$ and  this gives a Levi decomposition of
the parabolic subgroup  $P$. Since the Hermitian form  $h$ is the same
on  $X$ and  $V$, we  sometimes write  $U(X)$ and  $U(V)$ in  place of
$U(h_{\mid X})$ and $U(h)$. \\

The direct sum $W=Ev\oplus X$ has the property that the Hermitian form
on $W$  is {\it  degenerate} with $W^{\perp}=  Ev$.  The  quotient map
$W\ra W/Av=X$  preserves the Hermitian structure on  both sides, since
$Ev$ is  orthogonal to  $W$.  This induces  a surjective  map $U(W)\ra
U(X)$ of unitary groups, with  kernel $U_0$, say. Consider the abelian
vector group  $X^*=Hom (X,Ev)$.  We may  view $X^*$ as  a vector space
over $K$ and hence as the group of $K$- rational points of a unipotent
algebraic   group  which   is  isomorphic   to   $U_0$:  $U_0(K)\simeq
X^*$. Hence we refer to $U_0$ as a vector group. We have a split short
exact sequence
\[0  \ra   U_0  \ra  U(W)  \ra   U(X)  \ra  1,\]  and   we  may  write
$U(W)=U(X)U_0$.   An element $\alpha  $ of  $U(W)$ may  accordingly be
written as a pair $\alpha = (g,x)$ with $g\in U(X)$ and $x\in U_0$. If
$g\in U(X)$, then $g$ gives a transformation on the vector group $U_0$
defined by $x\mapsto xg$ (the  transpose of $g$).  With this notation,
if $\beta =(h,y)\in U(W)$ then
\[\alpha  \beta =  (gh,  xh+y). \]  Therefore, $\alpha  ^{-1}=(g^{-1},
-xg^{-1})$. \\

Suppose that  $m\in U(X)\subset U(W)$  is of the form  $(\lambda, 0)$,
which is the multiplication by the scalar $\lambda $ on $X$ and acting
by $0$ on  $v$. Given $a,b \in U(h)$ denote  by $[a,b]$ the commutator
$aba^{-1}b^{-1}$.

\begin{lemma}   \label{degcommutator}  If   $\alpha  =(g,x)$   and  $m
=(\lambda,0)$ then the commutator $[\alpha, m]$ is of the form
\[ [\alpha, m]= (1, (1-  \lambda ^{-1})xg^{-1}).  \] In particular, if
$x\neq 0$ and $\lambda \neq 1$, then the commutator $[\alpha, m]$ is a
non-zero element of the vector space $U_0$.
\end{lemma}

Recall that the abelian vector group  $X^*=Hom (X,Ev)$ may be viewed 
is a  vector space over  $K$ and hence  as the group of  $K$- rational
points of a unipotent  algebraic group $U_0$: $U_0(K)\simeq X^*$.  The
non-degenerate Hermitian form  $h_X$ on $X$ gives a  Hermitian form on
$X^*$ as well, which we again denote $h_X=h _{\mid X^*}$. \\

The group  $U=H(X)$ is non-abelian with  centre equal  to the 
one dimensional vector space ${\mathbb G}_a$ over $K$ and we have an 
exact sequence of unipotent $K$-algebraic groups
\[\{0\} \ra {\mathbb G}_a \ra U  \ra U_0 \ra \{1\}.\] This in turn
gives a short exact sequence
\[\{0\} \ra  K \ra U(K) \ra U_0(K)=X^*  \ra \{1\} ,\] at  the level of
$K$-rational  points (we  have  written $\{0\}$  and  $\{1\}$ for  the
trivial group  since one of them  is written additively  and the other
multiplicatively).  Moreover,  if $x,y\in U(K)$, and  $x^*,y^* $ their
images in $U_0(K)$,  then the commutator $[x,y]$, as  an element of $K$, 
is simply the imaginary part of $h(x^*,y^*)$. \\

Now, $P$ is a semi-direct product of $U$ with $M\simeq P/U$ and $M$ is
isomorphic  to  $U(X)$ as  in  the  preceding.   Moreover, the  centre
$Z={\mathbb G}_a$  of $U$ is normal  in $P$ and  secondly the quotient
$U/Z$ is isomorphic  to $U_0$.  We may write,  using the decomposition
$P=MU$, an element  $\alpha $ of $P$ in the  form $\alpha =(g,x)$ with
$g\in M$    and   $x\in   U$.     Let   $\lambda   \in
U(X)$ be a scalar transformation.

\begin{lemma} \label{nondegcommutator}

If $\alpha  = (g,x)$ with  $x\in U$, $x$  has non-zero image  in $U_0$
(i.e. does not  lie in the centre of $U$),  and $m=(\lambda ,0)\in M$,
then the  commutator $[\alpha, m]$  is a {\bf non-central}  element of
$U$. \\

\end{lemma}

The proof is immediate  from Lemma \ref{degcommutator} since the image
of the commutator $[\alpha ,m]$ in $U_0=U/Z$ is already non-trivial by
Lemma \ref{degcommutator}.

We now state another simple observation as a lemma.

\begin{lemma}  \label{finiteunipotent} If $U\ra  U_0$ is  the quotient
map, then a  subgroup $N\in U(O_K)$ has finite index  if and only if
its image $N_0$ has finite index in $U_0(O_K)$.
\end{lemma}

\subsection{An  inductive step  for integral  unitary groups}  

In  this subsection,  we prove  a  result which  will be  used in  the
inductive proof of Theorem \ref{mainth}.  This says that a subgroup of
the integral  unitary group which  contains finite index  subgroups of
smaller integral unitary groups, has finite index.

\begin{notation} Let $V=(V,h)$ be a nondegenerate Hermitian space over
$E$ such  that $K-rank  (U(h))\geq 2$. Let  $W,W'$ be  codimension one
subspaces  such that  the  restriction of  $h$  to $W$,  $W'$ and  the
intersection $W\cap  W'$ are all  non-degenerate.  We denote  by $U_V$
the unitary  group $U(h)$, and define  similarly $U_W,U_{W'}, U_{W\cap
W'}$. If  $Y$ is one of the  subspaces $W,W'$ and $W\cap  W'$, then by
the  non-degeneracy assumption,  we have  an  orthogonal decomposition
$V=Y\oplus Y^{\perp}$.  Hence $U_Y$ may  be thought of as the subgroup
of $U_V$ which acts trivially on $Y^{\perp}$. \\

Suppose that  $\Gamma \subset  U_V(O_K)$ is a  subgroup such  that its
intersection with  $U_W(O_K)$ has finite index in  $U_W(O_K)$ and such
that  its   intersection  with  $U_{W'}(O_K)$  has   finite  index  in
$U_{W'}(O_K)$.  Assume further  that  $W\cap W'$  contains a  non-zero
isotropic vector $v$.
\end{notation}

\begin{lemma} \label{inductivelemma} With  the preceding notation {\rm
(}and under the assumption  that $K$-$rank (U(h))\geq 2${\rm )}, The
group $\Gamma $ has finite index in $U_V(O_K)$.
\end{lemma}

\begin{proof} Since $W\cap  W'$ has codimension two in  $V$ and $h$ is
non-degenerate on $W\cap W'$ it follows  that $V$ is the direct sum of
$W\cap W'$ and its  orthogonal complement $(W\cap W')^{\perp}$ in $V$;
the orthogonal  complement is also a non-degenerate  unitary space (of
dimension two); by assumption, $W\cap W'$ contains an isotropic vector
$v$, say.   The non-degeneracy of $h$  on $W\cap W'$  shows that there
exists  a vector  $v^*\in W\cap  W'$  such that  $h(v,v^*)\neq 0$;  by
replacing $v^*$ by $v^*+\lambda v$  for a suitable scalar $\lambda$ if
necessary, we may  assume that $v^*\in W\cap W'$  is also an isotropic
vector.     Write    $V=(Ev+    Ev^*)\oplus    X$,    an    orthogonal
decomposition. Consider the filtration
\[0\subset Ev  \subset E\oplus X  \subset Ev\oplus X \oplus  Ev^*=V. \]
Denote  the  corresponding integral  Heisenberg  group (the  unipotent
subgroup of  $U(V)$ which preserves the  flag and acts trivially  on
successive  quotients), by $H_V=H (X)(O_K)$.  Define similarly,  the smaller
Heisenberg groups $H_W=H(X\cap W)(O_K)$ and $H_{W'}=H(X\cap W')(O_K)$. \\

By assumption, $H_W\cap \Gamma$ has finite index in $H_W$; similarly
for $H_{W'}$.  The two Heisenberg  groups generate $H_V$ up  to finite
index, since two distinct vector subspaces of codimension one span the
whole  space.   We thus  find  that the intersection of $\Gamma $ with 
the  integral  unipotent radical  of  a
parabolic $K$ subgroup $H_V=H(X)(O_K)$ has finite index in $H_V$. \\

Similarly, we  find a  finite index subgroup  of an  opposite integral
unipotent  radical  which lies  in  $\Gamma$;  therefore, $\Gamma$  is
arithmetic, by the Corollary \ref{criterion} to Theorem \ref{bamise}.
\end{proof}

\section{Properties of  the Burau representations 
$\rho  _n$ and $\rho _n(d)$} \label{burausection}

\subsection{Notation}   As   observed   in   the   introduction,   the
representation  $\rho_n:B_{n+1}\ra  GL_n(\Z[q,q^{-1}])$ preserves  the
Hermitian form

\[ h=h_n=  \begin{pmatrix} \frac{(q+1)^2}{q} &  -(1+q) & 0 &  \cdots &
\cdots\\ -(1+q^{-1}) & \frac{(q+1)^2}{q} & -(q+1) & \cdots & \cdots \\
0  & -(1+q^{-1}) &  \frac{(q+1)^2}{q} &  -(q+1) &  \cdots \\  \cdots &
\cdots & \cdots & \cdots & \cdots \end{pmatrix}. \]
(note that $h_{kk}$  is fixed under the involution $q\mapsto q^{-1}$). 
Denote by $D_n$ the determinant of $h_n$.

\begin{lemma} \label{squier} The determinant $D_n$ of the matrix $h_n$
is
\[D_n= det(h)=(\frac{q+1}{q})^n(\frac{q^{n+1}-1}{q-1}).\]
\end{lemma}

\begin{proof} Expanding  the determinant of $h_{n+1}$  using the first
row, we see that
\[D_{n+1}= \frac{(q+1)^2}{q}D_n-(1+q)(1+q^{-1})D_{n-1}.\]  Now an easy
induction implies the Lemma.
\end{proof}

\subsection{Non-degeneracy of the representation $(A^n, \rho _n(d))$ }

Consider the ring $A=R/(\Phi _d(q))\subset E=\Q [q]/(\Phi _d(q))$ ($q$
evaluated at {\it  all} the primitive root of unity).   We then have a
corresponding  Hermitian  form  on  $A^n$  which we  again  denote by
$h_n=h_n(d)$.  We will say that  elements $v_1, \cdots ,v_k$ are basis
elements of  $A^n$ if  there exist vectors  $v_{k+1}, \cdots,  v_n$ in
$A^n$ such  that $A^n$ is the  free module generated  by $v_1, \cdots,
v_n$.  We  have a  representation $\rho _n(d):B_{n+1}\ra  U(h) \subset
GL_n(A)$ as before. \\

\begin{lemma}  \label{isotropic}  Let  $A$  be as  in  the  preceding.
Assume that $d\geq 3$, so that $q\neq \pm 1$. Then

{\rm  [1]}   The  Hermitian form   $h_n$  on  the  module   $A^n$  is
non-degenerate  if and only  if $n$  is not  congruent to  $-1$ modulo
$d$. \\

{\rm [2]} If $d$ divides $n+1$, define $k$ by the formula $n= kd-1$.  
Then the vector
\[v=     e_1+(\frac{q^2-1}{q-1})e_2+     (\frac{q^3-1}{q-1})e_3+\cdots
(\frac{q^{kd-1}-1}{q-1})e_{kd-1}  \] is a  basis vector  and generates
the null  space of  the degenerate Hermitian  form $h$,  in $V_A=A^n$.
Moreover,  on   the  quotient  module   $V_A/Av$,  the  form   $h$  is
non-degenerate.  \\

{\rm [3]}  The vector  $v$ is  fixed by the  elements $s_j$  under the
representation  $\rho  _n(d)$.   We  therefore  get  the  ``quotient''
representation  $\overline{\rho _n(d)}$ of  $B_{n+1}$ on  the quotient
$V_A/Av$  which  again  preserves  the  {\rm  (}non-degenerate{\rm  )}
Hermitian form $h$ on $V_A/Av$.

\end{lemma}

\begin{proof} Part  [1] is obvious because of  Lemma \ref{squier}: the
determinant  of $h$  on $A^n$  vanishes if  and only  if $q^{n+1}-1=0$
where now  $q$ is  a primitive $d$-th  root of unity;  therefore, this
happens if and only if $n+1$ is divisible by $d$. \\

In  part  [2],  the  orthogonality  of  $v$  with  the  vectors  $e_j$
($j=1,2,\cdots,  kd-1$) follows  from  an explicit  computation; it  is
clear that $v,e_2,  \cdots, e_{kd-1}$ form a basis  of $A^{kd-1}$. The
matrix of the Hermitian form $h_{kd-1}$ on the quotient $A^{kd-1}/Av$
with  respect  to  the   basis  $e_2,  \cdots,  e_{kd-1}$  is  clearly
$h_{kd-2}$; this is non-degenerate by part [1]. Thus part [2] follows.

Since $v$ is orthogonal to $e_k$, it follows that $s_k$ fixes $v$: for
any $x\in A ^n$, we have
\[s_k(x)= x-\frac{qh(x,e_k)}{q+1}e_k,\] as can be easily checked
by  evaluating both  sides  of  this equality  on  the basis  elements
$e_l$. Therefore [3] follows.

\end{proof}

\begin{lemma}  \label{fullgroup}   Suppose  $n$  is   an  integer  not
congruent to $-1$  modulo $d$ and as before let  $\Gamma _n(d)$ be the
image of $B_{n+1}$ under the  representation $\rho _n(d)$ on $A^n$. If
$0\neq W \subset V_A=A^n$ is  a subgroup which is stable under $\Gamma
_n(d)$, then  $W$ contains $\lambda (A^n)$ for  some non-zero $\lambda
\in A$; in particular, $W$ has finite index in $A^n$. \\

If $F\supset  A$ is  a field containing  the integral domain  $A$, and
$W_F\subset F^n$ is a non-zero $F$-subspace stable under the action of
$\rho _n(d)$, then $W_F=F^n$.  
\end{lemma}

\begin{proof} Since $n\neq -1~(mod  ~d)$, $h$ is non-degenerate; since
$W\neq 0$,  $W$ has a nonzero vector  $x$, and there exists  a $k$ such
that $h(x,e_k)\neq 0$. The formula
\[s_k(x)=x-\frac{qh(x,e_k)}{q+1}e_k,\]  shows   that  a  nonzero
multiple $v_k$  of $e_k$ lies in  $W$; applying $s_k^{\pm  1}$ to this
multiple of  $e_k$ shows  that $(-q)^{\pm 1}v_k\in  W$ where $q$  is a
primitive root of unity; hence the $\Z[q,q^{-1}]$- module generated by
$v_k$ lies in $W$. Thus $Av_k\subset W$. \\

Applying  the elements  $s_{k\pm 1}$  to  $Av_k$ it  follows that  $W$
contains a non-zero multiple $v_l$ of $e_l$ for every $l\leq n$ and by
the preceding  paragraph, $Av_l\subset W$ for each  $l$. Therefore the
first part of the Lemma follows. \\

The second part  of the Lemma is proved in the  same way, by replacing
$A$ by the field $F$.

\end{proof}

\begin{proposition}  \label{burauirred} [1]  The  representation $\rho
_n(d)$ is irreducible unless $n\equiv -1 (mod~d)$. \\

[2] If $n\equiv  -1~(mod ~d)$ then $\rho _n(d)$  contains a vector $v$
fixed   under   $B_{n+1}$  and   the   quotient   $A^n/Av$  yields   a
representation  $\overline{\rho _n(d)}$  of $B_{n+1}$.   Restricted to
the  subgroup  of  $B_{n+1}$   generated  by  $s_2,\cdots,  s_n$,  the
representation   $\overline{\rho_n(d)}$   is   equivalent   to   $\rho
_{n-1}(d)$.  Hence $\overline{\rho_n(d)}$ is irreducible.\\

[3] Suppose  $n$ is  not congruent to  $-1$ modulo $d$.   Consider the
representation  $\rho  _n(d)$   over  an  algebraically  closed  field
${\overline E}$  containing the  ring $S=\Z[q+q^{-1}]$, and  denote it
$\rho _n({\overline E})$.  Then  $\rho_n ({\overline E})$ is irreducible
over ${\overline E}$.
\end{proposition}

\begin{proof} Part [1] of  the Proposition (the irreducibility) is the
second part of Lemma \ref{fullgroup}. \\

The second part is obvious. \\

The  third  part   again  follows  from  the  second   part  of  Lemma
\ref{fullgroup}: $\rho _n({\overline E})$ is irreducible.
\end{proof}

Let  $n\equiv  -1~(mod ~d)$.   Then  the  Hermitian space  $V=(A^n,h)$
contains  the  span $Av$  of  $v$  as  its orthogonal  complement  and
$\overline{V}=V/Av$ has a non-degenerate Hermitian form $\overline{h}$
induced by  $h$.  Let $U(h)$  and $U\overline{h})$ denote  the unitary
groups. We have then a split exact sequence of groups
\[\{0\}\ra  Hom_A(\overline{V},  Av)  \ra  U(h)\ra  U(\overline{V})\ra
\{1\}.\] As before, $\{0\}$ is  the trivial group  written additively
and $\{1\}$ is the trivial group written multiplicatively. \\

Here,  the dual  $W$ of  $\overline{V}$  may be  identified with  $Hom
_A(\overline{V},Av)$.   Thus $U(h)$  may be  written as  a semi-direct
product  $U(h)=U(\overline{h})W$ with  $W$  normal in  $U(h)$ and  the
conjugation action of $U(\overline{h})$ on $W$ is just the dual of the
standard representation of $U(\overline{h})$. 

\subsection{A central element of the braid group}

Consider the  braid group $B_{n+1}$ with  generators $s_1,s_2, \cdots,
s_n$. consider the element
\[\Delta=  (s_1s_2\cdots s_n)(s_1\cdots s_{n-1})\cdots(s_1s_2)(s_1),\]
of $B_{n+1}$.  It  is elementary to check that for  every $k\leq n$ we
have: $\Delta s_{n+1-k}=s_k\Delta$. Hence $\Delta ^2$ is in the centre
of $B_{n+1}$. \\

If $E$ is an algebraically closed field containing $\Z[q+q^{-1}]$, and
$n$ is not congruent to $-1$  mod $d$, then by part [3] of Proposition
\ref{burauirred}, $\rho  _n(E)$ is irreducible;  therefore, by Schur's
Lemma, the  element $\Delta ^2$ acts  by a scalar  $\delta$, say. Since
the  determinant of  each $s_i$  is $-q$  in the  representation $\rho
_n(E)$,   we   see   that   the   determinant  of   $\Delta   ^2$   is
$(-q)^{n(n+1)}=\delta ^n$. Therefore,  $(\delta /(-q)^{n+1})^n=1$. On the
other  hand, the  entries  of  $\rho _n(E)  (\Delta  ^2)$ are  Laurent
polynomials  in  $q$ with  integral  coefficients.   Hence the  scalar
$\delta  $ lies  in  the ring  $\Z[q,q^{-1}]$.   Therefore, the  above
equation means that $\rho _n(\Delta ^2)=\delta =(\pm 1)q^{n+1}$.

\begin{lemma}  \label{scalar}  If $d\geq  3$  and  $n=kd-2$, then  the
element $\Delta  ^2$ acts by  a scalar $\lambda  \neq 1$ on  the space
$A^n$ of the representation $\rho _n(d)$.
\end{lemma}

\begin{proof}  By  Proposition  \ref{burauirred},  the  representation
$\rho  _{kd-2}$ is  irreducible. By  the conclusion  of  the preceding
paragraph,  the   central  element  $\Delta  ^2$  acts   by  a  scalar
$\lambda=\pm  q^{n+1}=\pm q^{kd-1}  =\pm q^{-1}$  since $q$  is  now a
$d$-th  root of  unity. This shows that  $\lambda  \neq 1$  since,
otherwise, we get
\[q^{-2}=(\pm 1)^2=1.\] This is impossible  since $d\geq 3$ and $q$ is
a primitive $d$-th root of unity. Hence $\lambda \neq 1$. 
\end{proof}

\begin{remark} It can  be shown, by examining the  action of the Braid
Group (\cite{Bir})  on the  free group on  $n+1$ generators,  that the
scalar in  question is actually $q^{n+1}$  (this is a  special case of
Proposition \ref{Gassnerproposition} of the  present paper which we do
not prove). In particular, if  $n=kd-2$, and $q$ is a primitive $d$-th
root of unity, then $q^{n+1}= q^{-1}\neq 1$, even when $d=2$.
\end{remark}

We  have already  mentioned  in  the introduction  that  the proof  of
Theorem  \ref{mainth}   is  by  induction.   We   will  prove  Theorem
\ref{mainth} directly when $n\geq 2d$  is congruent to $-1$ (mod $d$).
Then we  will prove  that induction may  be applied, which  will prove
Theorem \ref{mainth} for all $n\geq  2d$.  To achieve this, we need to
exhibit sufficiently many unipotent elements in the image of $B_{n+1}$
under the representation $\rho _n(d)$ for the values $n\equiv -1 ~{\rm
or}~ 0$ (mod $d$).  This will be done in the next two subsections.

\subsection{Constructing  unipotent elements  when  $n=kd-1$.}  Assume
that $n=kd-1$.   Since  $dim X=kd-2$,  Lemma  \ref{scalar} and  Proposition
\ref{burauirred} imply that the square of the element
\[\Delta    '=   (s_2s_3\cdots    s_n)(s_2s_3\cdots   s_{n-1})(\cdots)
(s_2s_3)(s_2)\] lies in  the subgroup generated by $s_2, \cdots,
s_n$ in  $B_{n+1}$, and acts by a  non-trivial scalar $\lambda  \neq 1$ on
$\overline{\rho _{kd-1}(d)}$. \\

Denote the element $[s_1, \Delta '^2]$ of $B_{n+1}$ by $u$ (as before,
we denote by $[g,h]$ the  commutator $ghg^{-1}h^{-1}$ of $g$ and $h$).
The image of this element $u$ under the reducible representation $\rho
_{kd-1}(d)$, lies in the vector group $A^n$ where $U(h)=Hom _A(A^n,Av)
\rtimes U(\overline{h})$ is a semi-direct product as before.

\begin{proposition}   \label{unipotent}  Let   $n=kd-1$.    Under  the
representation $\rho _n(d)$  the commutator element $u=[s_1,(\Delta ')
^2]$ has  the property  that its image  $u'=\rho_{kd-1}(u)$ is  a {\bf
non-trivial} unipotent element in the vector part $A^{n-1}$ of $U(h)$.
\end{proposition}

\begin{proof} This  follows from Lemma  \ref{degcommutator}. Note that
$(A^n,h)$ is a degenerate Hermitian form with an isotropic vector $v$
such that the  quotient $V/Av$ is non-degenerate; moreover,  if $X$ is
the $A$-span of the vectors $e_2,  e_3, \cdots, e_n$ then $A^n$ is the
direct sum  $Av\oplus X$.  Hence  $(V,h)$ satisfies the  hypotheses of
Lemma \ref{degcommutator}. \\

Since  $Dim (X)=kd-2$,  Lemma  \ref{scalar} implies  that the  element
$m=\rho _n((\Delta ')^2)$ acts by a scalar $\lambda \neq 1$ on $X$ and
acts trivially on $v$; the  element $p=\rho _n(s_1)$ takes the element
$e_2\in X$ into the element
\[e_2+e_1= -qe_2+  v-\sum _{k=3}^n (\frac{q^k-1}{q-1})e_k  \notin X,\]
which  shows that  $p$ does  not  lie in  $M=U(X)$.  therefore,  Lemma
\ref{degcommutator}  applies, and  $u'$ is  a non-zero  vector  in the
vector part of  $U(h)$, and in particular, is  a non-trivial unipotent
element.
\end{proof}

\begin{proposition}  \label{unitaryunipotent}  Suppose that  $n=kd-1$.
The subgroup $U_n$ of $B_{n+1}$ generated by the conjugates $huh^{-1}$
where  $h$  runs  through  the  elements of  the  group  generated  by
$s_2,s_3, \cdots, s_n$ has  the property that under the representation
$\rho _n(d)$, it preserves the flag
\[Av\subset Av+Ae_2+  \cdots + Ae_{n-1} \subset A^n,\] and  acts trivially
on  successive quotients  of  this flag. \\

Further, if $U_0$  denotes the subgroup of $U(h)$  which preserves the
above flag and acts trivially  on successive quotients, then the image
of $U_n$  is a subgroup $N_0$  of finite index in  the integral points
$U_0(O_K)$.
\end{proposition}

\begin{proof} We need only prove the last part, since the image of $u$
lies in the  normal subgroup $U_0$ and is  preserved under conjugation
by elements of $U(h)$. The  conjugation action of $M$ on $U_0$ becomes
the  action  of $M=U(X)$  on  the dual  $X^*=Hom  (X,  Av)$. By  Lemma
\ref{fullgroup},  $N_0$ is  of  finite index  provided  it contains  a
non-zero  element of  $U_0$;  but it  was  already shown  (Proposition
\ref{unipotent}) that  the image of  $u$ is non-trivial in  $U_0$. The
proposition follows.
\end{proof}

\subsection{Constructing unipotents when $n=kd$ with $k\geq 1$} \label{nequals2dcase}
The space $V=A^n$ may be written as a direct sum
\[V=Av\oplus X  \oplus Av^*,~v= e_1+ \sum  _{k=2}^{n-1} x_ke_k, ~~v^*=
e_n+\sum _{k=1} ^{n-1}y_ke_{n-k}\] as  in the preceding. Then $M=U(X)$
is  the  Levi part  of  the parabolic  subgroup  $P$  of $U(h)$  which
preserves the flag
\[0\subset Av \subset Av\oplus X  \subset V,\] and $U$ is the subgroup
of $P$ which acts trivially on successive quotients of this flag.

Let  $\Delta  '$ be  as  in the  previous  subsection;  then $m=  \rho
_n((\Delta ')^2)$ lies in $M$ since it acts trivially on $v^*$ and $v$
($v^*$ and  $v$ are orthogonal to  all the vectors  $e_2, e_3, \cdots,
e_{n-1}$, and  hence are fixed by $s_2,  \cdots, s_{n-1})$; therefore,
$v,v^*$ are  fixed by $\Delta '$,  and $\Delta '$  preserves the space
$X$. By Lemma \ref{scalar} ,  $(\Delta ')^2$ acts by a scalar $\lambda
\neq 1$ on $X$. \\

\begin{proposition}\label{verifycriterion}  Assume  $n=kd$.  Then  the
following hold. \\

{\rm [1]}  The   element  $u=[s_1,  (\Delta  ')^2]$  acts   by  a  {\bf
non-trivial} unipotent element on $A^n$ under the representation $\rho
_n(d)$.   More  precisely,  the  element  $\rho  _n(d)  (u)$  lies  in
$U(O_K)\setminus Z$  i.e.  preserves the  flag $Av \subset  Av\oplus X
\subset V$, acts  trivially on the successive quotients,  and does not
preserve the subspace $X$. \\

{\rm [2]}  The group $N$  generated by the conjugates  $huh^{-1}$ with
$h\in \Gamma _{n-1}$ {\rm (}the  group generated by the elements $\rho
_n(s_2),  \cdots, \rho  _n(s_{n-1})${\rm )}  is a  subgroup  of finite
index in $U(O_K)$; in particular, $\Gamma _n$ intersects the unipotent
radical $U(O_K)$ in a subgroup of finite index.

\end{proposition}

\begin{proof}  The  element $u$  clearly  preserves  the  flag of  the
Proposition, as  was verified in the preceding  Proposition. Since $u$
acts by  a unipotent element on  $W$, and is a  commutator, it follows
that $u$ acts trivially on  the one dimensional space $V/W$, and hence
$u$ acts unipotently on $V$. \\

By the preceding proposition, $u$  does not preserve the space $X$: it
takes the basis element $e_2\in X$  into a sum of $av$ and elements of
$X$, with  $a\neq 0$ a  scalar.  By Lemma  \ref{nondegcommutator}, $u$
does not lie in the centre of $U$. \\

Let  $B$  be  the  group  generated  in  $U(O_K)$  by  the  conjugates
$huh^{-1}:h\in H$  where $H$ is  the group generated by  $s_2, \cdots,
s_{n-1}$. Under the quotient map  $U\ra U_0$, $B$ maps onto $B_0$, and
by the preceding Proposition, $B_0$ has finite index in $U_0(O_K)$. \\

Now    the    Proposition    follows,    by   appealing    to    Lemma
\ref{finiteunipotent}.

\end{proof}

\section{Proof     of     Theorem     \ref{mainth}} 
\label{proofssection}

\subsection{Proof of Theorem \ref{mainth}}

We  prove the  main  theorem (Theorem  \ref{mainth})  by induction  on
$n\geq 2d$;  we will prove  Theorem \ref{mainth} directly for  all $n$
which are multiples  of $d$ and are at least  $2d$. Then by induction,
Theorem \ref{mainth} follows. 

\subsubsection{Proof when $n=kd$, with $k\geq 2$}

The representation $\rho _{n-1}$  is not irreducible.  Let $Av \subset
V_{n-1}$ be  the subspace of invariants. The  quotient $V_{n-1}/Av$ is
an  irreducible  representation of  the  braid group $B(s_2,  \cdots,
s_{n-1})$.  Hence the commutator element
\[u=[s_1,\Delta (s_2,\cdots,  s_{n-1}) ^2 ] \] is  a {\it non-trivial}
unipotent element  under $\rho  _{n-1}$ and lies  in the  vector group
$Hom _A(V_{n-1}/Av, Av)$.  Therefore, $u$ preserves the flag $0\subset
Av\subset  V_{n-1}\subset V$ and  acts unipotently  on $V$  since $u$,
being  a   commutator,  has  determinant  one  and   is  unipotent  on
$V_{n-1}$. \\

Therefore, by  Proposition \ref{verifycriterion}, the  group generated
by  the  elements  $huh^{-1}$  with  $h\in  B(s_2,  \cdots,  s_{n-1})$
generate  a subgroup  commensurable  with $U(O_K)$  where  $U$ is  the
unipotent group preserving the  flag $Av \subset V_{n-1}\subset V$ and
acting  trivially  on  successive  quotients  of  the  flag.   We  have
therefore proved  that $\Gamma _n  \supset U(O_K)^N$ for  some integer
$N$. \\

Similarly, $\Gamma _n\supset U^{-}(O_K)^N$ for some integer $N$, where
$U^{-}$ is opposite  to $U$.  Therefore, $\Gamma _n$  is an arithmetic
group, since $n\geq 2d$ and  hence $U(V_n)$ has $K$-rank at least two:
the space $V_{d-1}$ spanned  by $e_1,\cdots, e_{d-1}$ has an isotropic
vector $v$  by Lemma \ref{isotropic}. The  subspace $V'_{d-1}$ spanned
by  $e_{d+1},  \cdots, v_{2d-1}$  also  has  the  same Hermitian  form
$h_{d-1}$   and   contains  an   isotropic   vector   $v'$  by   Lemma
\ref{isotropic}.   Clearly,  $V_{d-1}$  and  $V_{d-1}'$  are  mutually
orthogonal,  since the indices  $j$ of  the bases  $e_j$ differ  by at
least two.  Thus we  have produced two mutually orthogonal independent
isotropic  vectors,  and  hence  the  $K$-rank of  the  unitary  group
$U(V_n)$ is at least two.  Therefore, Theorem \ref{mainth} follows for
$n=kd\geq 2d$ from Corollary \ref{unitarycorollary}.

\subsubsection{Proof that  Theorem \ref{mainth} for $n-1$ implies  
Theorem \ref{mainth} for $n$ when $kd < n \leq kd+d-2$}

In this case, $\rho _n$  and $\rho _{n-1}$ are irreducible. Then $V_n$
contains  both  the  subspace  $V_{n-1}$  and the  span  $W_{n-1}$  of
$e_2,e_3, \cdots,  e_n$; moreover, both these  subspaces $V_{n-1}$ and
$W_{n-1}$ are non-degenerate  under $h$. The intersection $V_{n-1}\cap
W_{n-1}$ contains an isotropic vector, since the intersection contains
the span  of $e_2, \cdots,  e_{d+1}$: recall that $n\geq  2d$, hence
$n-2\geq d$.

By the  induction assumption,  there exists an  integer $N$  such that
$U(V_{n-1})^N   \subset   \Gamma_{n-1}    \subset   \Gamma   _n$   and
$U(W_{n-1})^N \subset \Gamma _{n-1}  \simeq <s_2, \cdots, s_n> \subset
\Gamma _n$.  It follows  from Lemma \ref{inductivelemma}  that $\Gamma
_n$ is arithmetic.

\subsubsection{Proof that  Theorem \ref{mainth} for $n-1$ implies 
Theorem \ref{mainth}  for $n$ when $n=kd-1$}.

By the  previous subsection, $\Gamma _{n-1}$ is  arithmetic. Since $n$
is congruent  to $-1$  modulo $d$, the Hermitian form 
$h_n$  is degenerate. Then, Proposition \ref{unitaryunipotent} implies that 
$\Gamma _n$  intersects the vector  part in an arithmetic  group. \\

Since the Hermitian form is  degenerate, the unitary group $U(V_n)$ of
$V_n$ is a semi-direct product  of its reductive part $U(V_{n-1})$ and
its  unipotent  part $Hom  (V_{n-1},  Av)$.   Then, the  decomposition
$U(V_n)(O_K)=U(V_{n-1})(O_K) Hom  _A(V_{n-1},Av)$ shows that $\Gamma  _n$ is
also arithmetic.\\

Combining the  above subsections  together, we obtain  a proof  of our
main result, namely Theorem \ref{mainth}, for all $n\geq 2d$. \\

\section{Proof of Theorems \ref{imagquad} and 
\ref{arithmetic}} 
\label{imagquadproofsection}

\subsection{Proof  of Theorem \ref{imagquad}}  In this  subsection, we
prove that if  $d$ is one of the numbers $3,4,6$  then for {\it every}
integer $n \geq 1$, the image of the Burau representation $\rho _n(d)$
is an arithmetic group.  First note that in these cases $d=3,4,6$, the
$d$-th cyclotomic extension $E_d=\Q[q]/(\Phi _d(q))\simeq \Q(e^{2\pi i
/d})$ is  an imaginary  quadratic extension of  $\Q$, and  the totally
real sub-field $K_d=\Q$;  the ring of integers $O_d$  is the ring $\Z$
of rational  integers.  Therefore,  we have $\Gamma  \subset U(h)(\Z)$
and  the ambient  Lie group  is  $U(h)(\R)$ since  $K\otimes _{\Q}  \R
=\R$. Note that  $U(h)(\R)$ is of the form $U(r,s)$  and there is only
one factor  involved, since $\Q$ has only  one archimedean completion,
namely $\R$. \\

We divide the proof into three cases. \\

Case 1. $n\geq 2d$. In this case, this is exactly Theorem \ref{mainth}
and this was already proved. \\

Case 2.  $n\leq d-2$  or $d\leq n \leq 2d-1$.  We refer  to Table 9 on
page 32 of \cite{Mc}. In  these cases, the group $SU(h)(\R)$ is either
compact or  is isomorphic  to $U(n-1,1)$. If  $U(h)$ is  compact, then
$\Gamma$ is  finite and is  hence ``arithmetic''; this is  the trivial
case. If  $U(h)=U(n-1,1)$, then by Theorem (10.3)  in \cite{Mc} (which
is actually a  special case of a result  of Deligne-Mostow), the image
$\Gamma =\rho _n(d)(B_{n+1})$ is a  lattice, i.e. $\Gamma $ has finite
index in $U(h)(\Z)$.  \\

Case 3.  $n= d-1$ or $n=2d-1$.   Since $n$ is congruent to $-1$ modulo
$d$,  Part [2]  of Proposition  \ref{burauirred} tells  us  that $\rho
_n(d)$ contains a trivial  sub-representation and that the quotient is
isomorphic to $\overline{\rho _{n-1}(d)}$; moreover, the unitary group
$U(h)=U(h_n)$  has  a unipotent  radical  $U_0$  and  a smaller  group
$U(h_{n-1})$  as a Levi  supplement. By  the preceding  paragraph, the
image  of $\Gamma $  in $U(h_{n-1})(\Z)$  is arithmetic;  moreover, by
Proposition  \ref{unitaryunipotent}  (note that  $n$  is  of the  form
$kd-1$  with  $k=1$ or  $2$),  the  image  of a  subgroup  $U_n\subset
B_{n+1}$ under the representation $\rho _n(d)$ is a subgroup of finite
index in the integral unipotent radical $U_0(\Z)$. Hence $\Gamma $ has
finite index in  the integral points of the  semi-direct product group
$U(h)$.

\subsection{Proof of Theorem \ref{arithmetic}} \label{proofarithmetic}
We  will  now prove  Theorem  \ref{arithmetic}.   Let ${\mathfrak  a}$
(resp.    ${\mathfrak   a}_e$)   denote   the   principal   ideal   in
$R=\Z[q,q^{-1}]$ generated  by the polynomial $(q^d-1)$  (resp. by the
$e$-th cyclotomic polynomial $\Phi  _e(q)$). Let $A= R/{\mathfrak a}$.
By the  Chinese remainder theorem,  the $\Q$-algebra $A\otimes  \Q$ is
the    product     ring    $A\otimes    \Q     =\prod    _{e\mid    d}
\Q[q,q^{-1}]/({\mathfrak  a}_e\otimes \Q)\simeq  \prod E_e$  where the
product runs  over all the divisors  of $d$. Here $E_e$  is the $e$-th
cyclotomic extension.  Consequently,  the reduced Burau representation
$\rho _n(A)$ on the rational vector space $A^n\otimes \Q \simeq E_e^n$
is  the direct sum  $\oplus _{e\mid  d} \rho  _n(e)$. Hence  the image
$\Gamma  $ of  the braid  group under  the Burau  representation $\rho
_n(A)$ lies in  the product $\prod U(h)(O_e)$ where  $O_e$ is the ring
of  integers in  the totally  real sub-field  $K_e=\Q({\rm  cos} \frac
{2\pi}{d})$, as in the introduction. \\

[1] We first  prove Theorem \ref{arithmetic} in the  case that $d$ and
$n+1$ are coprime.  In that case,  every divisor $e$ of $d$ is coprime
to $n+1$.  By Proposition  \ref{burauirred}, the Hermitian form $h$ on
$E_e^n$  is  non-degenerate and  the  representation  $\rho _n(e)$  is
irreducible.  Since $n\geq 2d$,  Theorem \ref{mainth} implies that the
image of $\Gamma $ under $\rho _n(e)$ is a subgroup of finite index in
$G_e(O_e)$, where $G_e=U(h)$ is the  unitary group of $h$ with respect
to  Hermitian  form  $h$  corresponding  to  the  quadratic  extension
$E_e/K_e$,  provided $e\geq  3$.  Moreover,  there exists  a subgroup,
namely  $SU(h)(O_e)$,  of  finite  index  in $G_e(O_e)$  which  is  an
arithmetic  subgroup of a  higher rank  semi-simple Lie  group, namely
$SU(h)K_e\otimes _{\Q} \R)$.  \\

If $e=2$ then $E_e=\Q[q]/(\Phi _2(q))=\Q$  and the image of $q$ in the
field $E_e=\Q$ is $-1$. The form $h$ vanishes identically, since $q+1$
divides all entries;  we replace this zero form  by first dividing $h$
by  $q+1$  for a  variable  $q$ and  then  taking  the resulting  form
evaluated at  $q=-1$.  This ``divided''  form is symplectic.Therefore,
$G_e$  is  the  symplectic  group;  then  by  the  result  of  A'Campo
\cite{A'C}, the image of $\Gamma $ in $G_e(O_e)=Sp(h, \Z)$ is a higher
rank arithmetic group. \\

If $e=1$ then  the Burau representation is evaluated  at $1$; i.e. the
representation  of the  braid group  $B_{n+1}$ lies  in  the symmetric
group  $S_{n+1}$  which  is a  finite  group  and  may be  ignored  in
questions on arithmeticity.  Lemma \ref{products} now implies that the
image  $\Gamma  $  in  $U(h)(A)=\prod  _{e\geq  2}  U(h)(O_e)$  is  an
arithmetic  group in  the product,  under the  assumption  that $n\geq
2d$. \\

[2] Suppose  $n+1$ and $d$ are not  coprime, and let $r\geq  2$ be the
g.c.d. of $d$ and  $n+1$. If a divisor $e$ of $d$  does not divide $r$
then  it does  not divide  $n+1$ either  and hence  $U(h)(O_e)$  is an
arithmetic  group in  a higher  rank semi-simple  group (up  to finite
index) and the projection to the $e$-th factor of the group $ \Gamma $
has finite  index in  $U(h)(O_e)$. \\ 

If $e$ does divide $n+1$, then $U(h)$ as an algebraic group over $K_e$
is not  reductive; suppose  $V_e$ is the  unipotent radical  of $U(h)$
viewed as a group over  $E_e$. By Theorem \ref{mainth} applied to this
case,  the projection of  $\Gamma $  to the  $e$-th factor  contains a
subgroup  of  finite index  in  $V_e(O_e)$  (see  the proof  of  Theorem
\ref{mainth}, the subsection where $n+1$ is divisible by $e$). \\

Putting the  above cases together, and using  Lemma \ref{products}, we
get Theorem \ref{arithmetic} in all cases.

\section{Relation with  a cyclic covering of  ${\mathbb P}^1$and proof
of theorem \ref{fullbraidmonodromy}}
\label{monodromyandburau}

In this section,  we relate the monodromy representation  of the braid
group $B_{n+1}$ considered  in Theorem \ref{fullbraidmonodromy} to the
Burau  representation   and  use   this  relation  to   prove  Theorem
\ref{fullbraidmonodromy}.

\subsection{Generalities} \label{general} Suppose 
\[\{1\}  \ra N  \ra F  \ra Q  \ra  \{1\} \]  is an  exact sequence  of
abstract  groups.  Suppose   $B$  is  a  subgroup  of   the  group  of
automorphisms of $F$ which stabilises the kernel $N$. Then $B$ acts on
the quotient $Q$ also by automorphisms, and acts on the abelianisation
$N^{ab}=N/[N,N]$; hence $B$ acts on the exact sequence
\[1  \ra N/[N,N] \ra  F/[N,N] \ra  Q \ra  \{1 \},  \] whose  kernel is
abelian. Moreover, the conjugation action of $F$ on the abelianisation
$N/[N,N]$ is trivial on $N$ and  descends to an action of the quotient
$Q$ on  $N^{ab}$. If  $N^{ab}$ is written  additively, then we  have a
$\Z[Q]$ module structure on $N^{ab}$.   (Further, the action of $Q$ on
$N^{ab}$  is via  its action  on the  group  $Out(N)=Aut(N)/Int(N)$ of
outer automorphisms of $N$). \\

If we now assume that $B$  acts trivially on the quotient $Q$, we then
have an action of the  product group $B\times Q$ on the abelianisation
$N^{ab}$.   Therefore, the  action of  $B$ on  $N^{ab}$ is  by $\Z[Q]$
module endomorphisms.

\subsection{Action    of    the    braid    group    on    the    free
group} \label{artinaction}

Suppose  $F_{n+1}$ is  a  free group  on  $n+1$ generators,  $x_1,x_2,
\cdots, x_{n+1}$.   Recall that the  braid group on $n+1$  strands was
given by generators  $s_i$ and relations given in  the introduction. A
theorem of  E.Artin (\cite{Bir}, Corollary 1.8.3) says  that the  
braid group $B_{n+1}$  acts on  $F_{n+1}$ as
follows. If $j\leq  i-1$ or if $j \geq i+2$,  then $s_i(x_j)= x_j$; If
$j=i,i+1$ then the action is
\[s_i(x_i)=x_{i+1},  ~~s_i(x_{i+1})=  x_{i+1}^{-1}x_ix_{i+1}.  \]  (In
\cite{Bir}, Corollary  1.8.3, the action is  on the right;  to get the
left action, one can make a slight modification of the formulae). \\

The  action of $B_{n+1}$  on $F_{n+1}$  gives an  action of  the braid
group on the abelianisation $F_{n+1}^{ab}=\Z^{n+1}$ of the free group;
the images  of $x_i$ form the  standard basis of  $\Z^{n+1}$. From the
equations in  the preceding  paragraph, it is  clear that  the element
$s_i$ acts by  the permutation matrix interchanging $i$  and $i+1$ and
fixing the  rest of  the basis.  It  is also  clear that the  image of
$B_{n+1}$ in the automorphisms of $\Z ^{n+1}$ is the permutation group
$S_{n+1}$ on $n+1$ symbols. \\

The  kernel, denoted  $P_{n+1}$, of  the map  $B_{n+1}\ra  S_{n+1}$, is
called the {\it pure braid group}.

\subsection{Realisation of the Burau representation on homology}

Write $G=\Z^{n+1}$ for the abelianisation of the free group $F_{n+1}$,
where $G$  is written multiplicatively.  We have a map  $G\ra \Z\simeq
q^{\Z}$ given by
\[x_1^{m_1}x_2^{m_2}\cdots                     x_{n+1}^{m_{n+1}}\mapsto
q^{-(m_1+m_2+\cdots+m_{n+1})}.\]  The group  $S_{n+1}$ (and  hence the
braid  group $B_{n+1}$)  acts  on $\Z^{n+1}$  by  permutations on  the
standard  basis, and  acts trivially  on  $q^{\Z}$. The  above map  is
equivariant with  respect to this  action.  Moreover, the  braid group
$B_{n+1}$  acts on  $F_{n+1}$  and the  map  $F_{n+1}\ra \Z^{n+1}$  is
equivariant with respect to this action. We now have an exact sequence
\[ 1 \ra  K_{n+1}\ra F_{n+1}\ra q^{\Z} \ra 1  \] with kernel $K_{n+1}$
stable under the action of the braid group $B_{n+1}$. \\

The  group $F_{n+1}$  has  the standard  generators $x_1,x_2,  \cdots,
x_{n+1}$. {\it As a normal subgroup of} $F_{n+1}$, the group $K_{n+1}$
is s generated by the elements
\[y_1=x_1^{-1}x_2,  y_2=x_2^{-1}x_3,\cdots,  y_n=x_n^{-1}x_{n+1}.\] As
was observed in the  preceding paragraph, $B_{n+1}$ acts on $K_{n+1}$;
we compute its action  on the ``standard basis'' $y_1,y_2, \cdots,y_n$
of $K_{n+1}$. The braid group is  generated by $s_i$ and the action of
$s_i$   on   $F_{n+1}$   was   described   before.   Hence   we   have: 
$s_i(y_{i-1})=x_{i-1}^{-1}x_{i+1}=y_{i-1}y_i$,
\[s_i(y_j)=s_i(x_{j-1}^{-1}x_j)= y_j~~(j\leq i-2, \quad {\rm or} \quad
j\geq i+2),\]
\[s_i(y_i)=x_{i+1}^{-1}x_{i+1}^{-1}x_ix_{i+1}=x_{i+1}^{-1}y_i^{-1}x_{i+1}=
^{x_{i+1}^{-1}}(y_i^{-1})\] and
\[s_i(y_{i+1})=          x_{i+1}^{-1}x_i^{-1}          x_{i+1}x_{i+2}=
x_{i+1}^{-1}x_i^{-1}x_{i+1}^2y_{i+1}=    ^{x_{i+1}^{-1}}(y_i)y_{i+1}.\]
We now consider the  commutator subgroup $K_{n+1}^{(1)}$ of $K_{n+1}$;
it is a  normal subgroup of $F_{n+1}$ and is  stabilised by the action
of the braid group.  Therefore, we have an exact sequence of groups
\[1 \ra M_n= K_{n+1}/K_{n+1}^{(1)}\ra F_{n+1}/K_{n+1}^{(1)} \ra q^{\Z}
\ra  1, \]  with abelian  kernel $M_n$  (namely the  abelianisation of
$K_{n+1}$)  written additively.  Since $K_{n+1}$  is generated  by the
$y_i$ as a normal subgroup  of $F_{n+1}$ and the conjugation action of
$F_{n+1}$  on  the abelian  group  $M_n$  descends  to the  action  of
$q^{\Z}$, it follows that $M_n$  is generated as a $q^{\Z}$-module, by
the images $y'_i$ (under the  quotient map $K_{n+1}\ra M_n$) of $y_i$.
Since the group law on $M_n$ is written additively, for each $x_i$ the
conjugation action  on $M_n$ is  simply multiplication by  the element
$q^{-1}$.  We  are in the  situation of subsection  \ref{general} with
$N^{ab}=M_n$, $Q=q^{\Z}$ and $B=B_{n+1}$.  \\

The  action  of the  braid  group  on the  basis  $y_i$  of the  group
$K_{n+1}$ was computed  above; this gives a description  of the action
of $s_i$ on the basis $y_i'$ of $M_n$ as follows. We have the formulae
\[s_i(y_j')=y_j'~~(j\leq i-2  \quad {\rm  or} \quad j\geq  i+2), \quad
s_i(y_{i-1}')= y_{i-1}'+y_i', \]
\[s_i(y_i')= -qy_i' , \quad s_i(y_{i+1}')=y_{i+1}'+qy_i'.\]

Now write $y_i= q^ie_i$. In the basis $\{e_i\}$, we get 
\[s_i(e_j)=e_j~(\mid j-i\mid \geq 2), \quad s_i(e_{i-1})= e_{i-1}+qe_i,\]
\[s_i(e_i)=-qe_i,  \quad s_i(e_{i+1})=e_{i+1}+e_i,\] which  is exactly
the  reduced  Burau representation  defined  in  the introduction.  We
therefore have:

\begin{theorem}  \label{burautheorem} {\rm  (Burau)} Let  $K_{n+1}$ be
the kernel  of the  map $F_{n+1}\ra q^{\Z}$  defined above.  Then, the
action of the braid Group $B_{n+1}$ on the first homology of $K_{n+1}$
with  integral  coefficients,  is  isomorphic  to  the  reduced  Burau
representation.
\end{theorem}

\subsection{Realisation of the Burau representation at $d$th roots of unity} \label{buraurealise}

Consider  now  the   quotient  map  $q^{\Z}  \ra  q^{\Z}/q^{d\Z}\simeq
\Z/d\Z$.    We   have   a   surjective   map   $F_{n+1}\ra   q^{\Z}\ra
q^{\Z}/q^{d\Z}$. This defines an exact sequence
\[1\ra  K_{n+1}(d)\ra  F_{n+1}\ra  \Z/d\Z  \ra 1,\]  of  groups,  with
$Z/d\Z$   written   multiplicatively.    Clearly,  $K_{n+1}(d)$   is
generated ( as  a {\it normal}  subgroup of  $F_{n+1}$) by the elements
$y_1,y_2,    \cdots,   y_n$    and   $x_1^d$    where,    as   before,
$y_i=x_i^{-1}x_{i+1}$.   Therefore, $K_{n+1}(d)$  is generated  by the
elements $x_1^d$  and the collection  $\{x_1^jy_ix_1^{-j}: 0\leq j\leq
d-1,~i\leq n\}$.

Being a subgroup  of $F_{n+1}$, $K_{n+1}(d)$ is also  free and if $n'$
denotes the minimal  number of generators of this  free group, we have
the formula
\[1-n'=d(1-(n+1))=-dn,~~n'=  1+dn.\]  Since  the  cardinality  of  the
system of generators  we have exhibited is exactly  $1+nd$, it follows
that the above generators $x_1^d$ and $x_1^jy_ix_1^{-j}$, {\it freely}
generate $K_{n+1}(d)$.\\

Therefore     (subsection    \ref{general}),     the    abelianisation
$K_{n+1}(d)^{ab}$ is a direct sum of a {\it free} module over the ring
$\Z[G]=\Z[q]/(q^d-1)=A $ with  generators $y_1',y_2',\cdots, y_n'$ and
the trivial module  $\Z (x_1 ^d) '$. Here the  prime denotes the image
of the element under  the quotient map $K_{n+1}(d)\ra K_{n+1}(d)^{ab}$.  
The  map $K_{n+1}^{ab}\ra K_{n+1}(d)^{ab} $  is equivariant for
the action of the braid group.  We have an exact sequence of $B_{n+1}$
modules:
\[0  \ra Image  (K_{n+1}^{ab})  \ra K_{n+1}(d)^{ab}  \ra  \Z \ra  0.\]
Therefore, it  follows from Theorem \ref{burautheorem}  that the braid
group acts on the abelianisation $K_{n+1}(d)^{ab} $ and that
the latter is  an extension of the reduced  Burau representation $\rho
_n  (A)$  by the  trivial  representation.   It  follows from  Theorem
\ref{arithmetic}  that the  image of  the representation  $
B_{n+1}\ra GL(K_{n+1}(d)^{ab})$ is an arithmetic group.

\subsection{Some cyclic coverings of ${\mathbb P}^1$}

Let $a_1,a_2, \cdots,a_{n+1}$ be distinct complex numbers; write $S_a$
for     the     complement     in     $\C$    of     these     points:
$S_a=\C\setminus\{a_1,a_2,\cdots,a_{n+1}\}$. The  fundamental group of
$S_a$, once  a base point is  chosen, may be identified  with the free
group on $F_{n+1}$  generated by small circles $x_i$  going around the
point $a_i$  counterclockwise once (and  joined to the  preferred base
point by an  arc which avoids all the other points  $a_j$ and has zero
winding number  around all the points  $a_j$ with $j\neq  i$). The map
$S_a\ra \C^*$ defined by
\[x\mapsto    (x-a_1)(x-a_2)\cdots(x-a_{n+1})=P_a(x),\]    induces   a
homomorphism $F_{n+1}\ra q^{\Z}$, which  sends each $x_i$ to $q^{-1}$.
Here,  $q^{-1}$ is a  small circle  around zero  in $\C^*$  which runs
counterclockwise exactly once.\\

For  future reference,  note that  the loop  around infinity  lying in
$S_a$ represents  the product element $x_1x_2\cdots  x_{n+1}$ and that
this  element  is  invariant  under  the action  of  the  braid  group
$B_{n+1}$ on the free group $F_{n+1}$. \\

The affine  variety $\C^*={\mathbb G}_m$  admits a cyclic  covering of
order $d$  given by $z\mapsto  z^d$ from ${\mathbb G}_m$  to ${\mathbb
G}_m$.   The  covering  may  be  realised  as  the  space  $\{(x,y)\in
\C^*\times  \C^*:y^d=x\}$   where  the  covering  map   is  the  first
projection.   Pulling back  this covering  to  $S_a$ we  get a  cyclic
covering of $S_a$, realised as the space
\[X_a=\{(x,y)\in                    \C^*\times                    S_a:
~~y^d=(x-a_1)(x-a_2)\cdots(x-a_{n+1})\},\]  with the  first projection
being the  covering map  from $X_a$ onto  $S_a$. Therefore,  under the
identification of  the fundamental group of $S_a$  with $F_{n+1}$, the
fundamental group of $X_a$ is identified with $K_{n+1}(d)$. \\

As  the collection $a$  varies, we  get a  collection $\mathcal  P$ of
monic polynomials $P_a$ of degree $n+1$ which have distinct roots, and
if ${\mathcal Q}$ denotes the variety
\[(w,x,P)\in \C^*\times  \C \times {\mathcal  P}: w= P(x),\]  then the
projection on to the third coordinate gives a fibration over $\mathcal
P$ with fibre at $P$ being  $S_a$ (here $a$ is the collection of roots
of $P$).  We therefore get a monodromy action of the fundamental group
of  ${\mathcal P}$  as outer  automorphisms of  the  fundamental group
$F_{n+1}$ of the fibre.  We  have the following fundamental theorem of
E.  Artin (\cite{Bir}, 1.8) .

\begin{theorem}  {\rm   (}  Artin  {\rm)}   \label{artinstheorem}  The
fundamental group of $\mathcal P$ is the braid group $B_{n+1}$ and the
monodromy action on $F_{n+1}$ is  the action of $B_{n+1}$ on $F_{n+1}$
defined in subsection \ref{artinaction}.
\end{theorem}

Consequently, the monodromy action on the fibre of the fibration
\[\{(y,x,P)\in  \C^*\times  \C\times  {\mathcal  P}:  y^d=P(x)\]  over
${\mathcal  P}$  is the  action  of  $B_{n+1}$  defined in  subsection
\ref{artinaction}  restricted to  the  subgroup $K_{n+1}(d)\simeq  \pi
_1(X_a)$; therefore,  $B_{n+1}$ acts on  the first homology  of $X_a$:
$H_1(X_a)\simeq  K_{n+1}(d)^{ab}$  by  (an  extension by  the  trivial
representation of )  the Burau representation $\rho _n(A)$  , and this
gives  the monodromy  action of  $B_{n+1}$ on  $H_1(X_a)$,  with image
$\Gamma '$, say.

\begin{theorem} \label{openmonodromy} If $n\geq 2d$, then the image of
the representation $B_{n+1} \ra GL(H_1(X_a))$ is an arithmetic group.
\end{theorem}

\begin{proof}  We   have  identified  this   representation  with  the
extension by  the trivial  representation of the  Burau representation
$\rho _n(A)$ where $A=\Z[q,q^{-1}]/(q^d-1)$.  The Theorem follows from
the conclusion of the preceding subsection \ref{buraurealise}.
\end{proof}

\subsection{The  compactification of  $X_a$}  Now $X_a$  is a  compact
Riemann surface  with finitely many  punctures; denote by  $X_a^*$ the
smooth projective curve obtained by filing in these punctures. \\

The covering map  $X_a \ra S_a$ is such that  these punctures lie over
the points  $a_i$ or else  over the point  at infinity of $S_a$.  If a
puncture lies over  some $a_i$, then the image of  a small loop around
the  puncture in  $F_{n+1}$ is  $x_i^d$;  if the  puncture lies  above
infinity,  then the  image  of a  small  loop around  the puncture  in
$F_{n+1}$   is  a   power  of   the  element   $x_1x_2\cdots  x_{n+1}$
(represented by the loop  around infinity); therefore, such an element
is invariant under the action of the braid group.\\

The mapping of $\pi _1(X_a)\ra \pi _1(X_a^*)$ is such that these loops
around the punctures generate the  kernel (this is an easy consequence
of the van  Kampen theorem); note that the element  $s_j$ of the braid
group $B_{n+1}$ (under the action  on the free group $F_{n+1}$ defined
in subsection \ref{artinaction}), takes  the loop $x_i$ to a conjugate
of the loop $x_k$ for some  $k$. Therefore, the braid group leaves the
kernel of  the map $\pi  _1(X_a)\ra \pi _1(X_a^*)$ stable.   Hence, by
Artin's   theorem  (Theorem   \ref{artinstheorem})  the   induced  map
$H_1(X_a)\ra  H_1(X_a  ^*)$ on  homologies  is  equivariant under  the
monodromy action of the braid group. \\

\subsection{Proof                      of                      Theorem
\ref{fullbraidmonodromy}}  \label{prooffullbraid}  We  can  now  prove
Theorem \ref{fullbraidmonodromy}.  We are to  prove that the  image of
the  representation $B_{n+1}\ra  GL(H^1(X_a^*))$  is arithmetic  where
$H^1(X_a^*)$  is  the   cohomology  with  integer  coefficients.   The
$B_{n+1}$ module $H_1(X_a^*)$ ({\it homology} of $X_a^*$ with integral
coefficients)  is a  quotient of  the module  $H_1(X_a)$.   By Theorem
\ref{openmonodromy} the image of  the braid group in $GL(H_1(X_a))$ is
arithmetic.   By   Proposition  \ref{arithmeticimageproposition},  the
image of the  braid group in $GL(H_1(X_a^*))$ is  also arithmetic.  By
Poincar\'e duality, the image of  the braid group in $GL(H^1(X_a))$ is
also arithmetic, proving Theorem \ref{fullbraidmonodromy}.

\subsection{The representation $H_1(X_a^*,\Q)$}  

Consider     the      $\Q$     algebras     $A=\Q[q]/(q^d-1)$,     and
$A'=\Q[q]/(1+q+\cdots + q^{d-1})$. If $M$  is an $A$ module, denote by
$(1+q+\cdots  +  q^{d-1})M$, the  subspace  of  elements  of the  form
$(1+q+\cdots+q^{d-1})m$  with $m\in M$  and let  $M'$ be  the quotient
module    $M/(1+q+\cdots    +q^{d-1})M$. \\

We have  seen that $K_{n+1}(d)^{ab}$ is  an extension of  the image of
$K_{n+1}^{ab}$  in  $K_{n+1}(d)^{ab}$,  by  the trivial  module  $\Z$;
tensoring with $\Q$, we have the same statement, with $\Z$ replaced by
$\Q$.   Clearly, $\Q  '=0$.  Therefore,  we have  the equality  of the
``primed'' modules
\[K_{n+1}^{ab}\otimes  \Q\simeq Im  (K_{n+1}^{ab}\otimes  \Q)'. \]  We
have   seen   in  the   preceding   subsection   that  the   ``primed''
representation  $H_1(X_a,\Q)'$   of  the  braid   group  $B_{n+1}$  is
isomorphic to the Burau  representation $\rho _n(A')$ on $A'^n$, where
$A'=\Q[q,q^{-1}]/(1+q+\cdots  +q^{d-1})$.   Therefore,  by  subsection
\ref{proofarithmetic}, the  Burau representation $\rho  _n(A')$ on the
$\Q$-vector space $(A ')^n$ is the direct sum
\[H_1(X_a,\Q)'\simeq  \rho _n(A')=  \oplus  _{e \mid  d~~e\geq 2}  \rho
_n(e).\]  Recall from  Proposition \ref{burauirred}  that  if $n\equiv
-1~(mod  ~ e)$  then the  representation $\rho  _n(e)$ contains  a one
dimensional  space $L_e$, say,  of invariants,  and that  the quotient
$(A_e^n\otimes \Q)/L_e$    is    irreducible;     in    Proposition
\ref{burauirred},  this  representation  was  denoted  $\overline{\rho
_n(e)}$.  By  an abuse of notation,  if $e$ does not  divide $n+1$, we
denote $\rho _n(e)$ also by $\overline{\rho _n(e)}$.  \\

The  embedding of  the affine  curve $X_a$  into  its compactification
$X_a^*$  induces  a  map  $H_1(X_a,\Q)\ra H_1(X_a^*,\Q)$  on  rational
homology, which is equivariant for  the action of the braid group.  By
analysing the map $H_1(X_a)\ra H_1(X_a^*)$, one can show that (compare
Theorem  5.5  of \cite{Mc}  on  pp.  24-25  where  the  case $e=d$  is
treated)   the  monodromy   representation  $H_1(X_a^*,\Q)$   has  the
decomposition
\[H_1(X_a^*,\Q) \simeq \bigoplus _{e \mid d, ~ e\geq 2}{\overline \rho
_n(e)}.\] We use the fact that  loops around points in $X_a$ which lie
above  $\infty$ in  the curve  $S_a\subset {\mathbb  P}^1$ lie  in the
kernel of  the map  $H_1(X_a) \ra H_1(X_a^*)$;  one can show  that the
invariant vector in $\rho _n(e)$ (for $e$ dividing $n+1$) is generated
by these ``infinity'' loops and hence lies in the kernel of the map on
homology. \\

The  following   Proposition  is  an  immediate   consequence  of  the
decomposition of  the representation of $B_{n+1}$  on $H_1(X_a^*)$ and
Lemma \ref{products}.

\begin{proposition}  \label{arithmeticspecify} 
If $n\geq 2d$,  then the  image of  the monodromy
representation    of   $B_{n+1}$    on    $H_1(X_a^*)$   of    Theorem
\ref{fullbraidmonodromy} is a subgroup  of finite index in the product
$\prod  {\overline  G}_e(O_e)$  where  the  product is  over  all  the
divisors $e\geq 2$  of $d$, ${\overline G_e}$ is  the unitary group of
the  Hermitian  form  ${\overline  h_n}$  induced by  $h=h_n$  on  the
quotient  representation   ${\overline  \rho  _n(e)}$   of  the  Burau
representation $\rho _n(e)$.

\end{proposition}

\section{Applications} \label{applications}

\subsection{Some complex reflection groups} \label{mcmullen}

We will follow  the notation of Section 5 of  \cite{Mc}.  In Section 5
of \cite{Mc},  given the root  system $A_n$ (and therefore  its graph)
the  Artin  group $A(A_n)$  is  defined;  given  a  complex  number
$q=e^{2\pi  ix}$   with  $-1/2   \leq  x  <   1/2$,  there   exists  a
representation
\[\rho  _q: A(A_n) \ra  GL_n(\C),\] with  image denoted  $A_n(q)$. The
image preserves a Hermitian form and  is a subgroup of a unitary group
$U(r,s)\subset  GL_n(\C)$.  The  image  of the  braid group  $B_{n+1}$
under    the   Burau    representation   $\rho    _n:    B_{n+1}   \ra
GL_n(\Z[q,q^{-1}])$  may  be identified  with  the complex  reflection
group  $A_n(q)$.  Question  5.6  of \cite{Mc}  asks  when the  image
$A_n(q)$ is a lattice in  $U(r,s)$.  In the notation of question 5.6
of \cite{Mc}, the image group  $\Gamma _n= \rho _n(d)(B_{n+1})$ is the
group  $A_n(q)$  where  $q$  is  a primitive  $d$-th  root  of  unity;
therefore, question  5.6 asks whether  $A_n(q)$ can be a  lattice in
the  real unitary  group $U(r,s)$;  Theorem \ref{mainth}  answers this
question in a  large number of cases. We  have the following Corollary
of Theorem  \ref{mainth} (and part  [1] of the Corollary  follows from
Theorem \ref{imagquad}).

\begin{corollary} 

[1] If  $q=e^{2\pi i/d}$ then $A_n(q)$  is a lattice  in $U(r,s)$ when
$d=3,4,6$ for  {\bf all n}.  If $d=2$, then  $A_n(q)$ is a  lattice in
$Sp_{n}({\mathbb Z})$.\\

[2] If $d$ is not $2,3,4,6$ and  $n\geq 2d$ , then the image $\Gamma _n$
under  the  Burau representation  is  an  irreducible  lattice in  the
product  of unitary  groups $U(h)(K_d\otimes  _{\Q}\R)\simeq  \prod _v
U(r_v,s_v)$  where the  product  is over  all  the archimedean  (real)
completions $K_{d,v}$ of the totally  real field $K_d$, and the number
of factors is at least  two.  Therefore, the projection of $\Gamma _n$
to one of the factors is {\it never} a lattice if $d\neq 3,4,6$. \\

In particular,  the intersection $A_n(q)\cap SU(r_v,s_v)$  is dense in
$SU(r_v,s_v)$ for each archimedean $v$.
\end{corollary}

Thus question  5.6 of  \cite{Mc} is open  only if $d\neq  3,4,6$ {\it
and} if $n\leq 2d-1$. 

\subsection{Application      to     monodromy     of      type     $_n
F_{n-1}$} \label{sarnak}

\begin{defn}

Let $\alpha  =(\alpha _1,  \cdots,\alpha_n )\in \C  ^n$ and  $\beta =(
\beta _1,  \cdots, \beta  _n )\in \C^n$  be complex numbers  such that
$\alpha _j \neq \beta _k~({\rm mod}~~1)$ for any $j$ and $k$. \\

Denote by $z$  a  complex   variable; write  $\theta  =z\frac{d}{dz}$. 
Consider the differential operator $D=D(\alpha_1, \cdots,
\alpha _n, \beta _1, \cdots \beta _n)$ given by
\[D= (\theta +\beta _1-1)(\theta +\beta _2 -1)\cdots (\theta +\beta _n-1) -
z(\theta +\alpha _1)\cdots (\theta +\alpha _n). \]

The equation 
\[Du=0,\] is called the  {\it hypergeometric equation} with parameters
$\alpha , \beta$.
\end{defn}

The differential operator $D$ is of the form 
\[D=a_n(z)\frac{d^n}{dz^n} +a_{n-1}(z)\frac{d^{n-1}}{dz^{n-1}}+\cdots +
a_0(z),\]  where $a_i$  are polynomials  and $a_n(z)=z^n(1-z)$  is the
highest coefficient.  \\ 

This  coefficient  $a_n(z)$ vanishes  at  $0$  and  $1$ (and  also  at
infinity if we change coordinates from  $z $ to $z^{-1}$) but does not
vanish anywhere in $C=  {\mathbb P}^1\setminus \{0,1,\infty \}$. Using
this property  of the  highest coefficient, it  can be shown  that the
space of  solutions is  of dimension $n$,  and that the  solutions are
(locally)   analytic   on  the  Riemann surface  $C$.   Denote  by   $\pi
_1(C,\frac{1}{2})$, the  fundamental group of  $C$ based at  the point
$\frac{1}{2}$.  Then it acts on the space ($\simeq \C^n$) of solutions
$u$ of the foregoing differential  equation $Du=0$ ($u$ is analytic on
$C$).  Denote  by $\Gamma =\Gamma (\alpha,  \beta) $ the  image of the
fundamental group  $\pi _1(C,\frac{1}{2})$ in the  group $GL_n(\C)$ of
linear  automorphisms  of  the  $n$-dimensional  space  of  solutions.
Denote  by $M_{\alpha,  \beta}$ the  resulting representation  of $\pi
_1(C, \frac{1}{2})$.  \\

If  $\alpha, \beta \in  \Q^n$ then  $\Gamma $  may be  conjugated into
$GL_n(O)$ where $O$ is the ring of integers in a number field $F$.  If
$D=[F:\Q]$  is the  degree  of  $F$ over  $\Q$,  then $\Gamma  \subset
GL_{nD}(\Z)$.   In \cite{Sar}  the following  question  is considered:
determine the pairs $\alpha, \beta \in \Q ^n$ such that the associated
monodromy group is arithmetic (i.e.   has finite index in its integral
Zariski closure in $GL_{nD}(\Z)$).\\

Fuchs, Meiri  and Sarnak (see  \cite{Sar}) give an infinite  family of
examples of pairs $\alpha, \beta$ for  which the group $\Gamma $ has a
natural  embedding  in an  integral  orthogonal  group $O_n(\Z)$  with
Zariski dense image  and of {\it infinite index}  (In \cite{Sar}, they
are called  {\it thin} groups).  Thus  the monodromy $\Gamma  $ is not
always arithmetic. \\

By  using   \cite{A'C},  an  infinite   family  of  examples   can  be
constructed,  where  the  monodromy   is  an  arithmetic  subgroup  of
$Sp_{2m}(\Z)$.  One  can,  using  Theorem \ref{mainth},  give  a  more
general formulation: let $d\geq 2$ be  an integer and $1\leq c \leq d$
be an integer coprime to $d$.

\begin{theorem}  \label{beukers} Suppose that 
\[\alpha=(\frac{c}{d}+\frac{1}{n+1},\cdots, \frac{c}{d}+\frac{n-1}{n+1}, 
\frac{c}{d}+\frac{n}{n+1}),\] 
\[\beta =(~\frac{c}{d}+\frac{1}{n},\cdots , \frac{c}{d}+\frac{n-1}{n},~1).\] 

If $d=2$, and  $n=2m$ is even, then the monodromy  group $\Gamma $ (is
an arithmetic group  and) has finite index in  the integral symplectic
group $Sp_{2m}(\Z)$. \\

If $d\in\{3,4,6\}$ and  $n$ is arbitrary such that  $n+1$ is coprime to
$d$, then  the monodromy $\Gamma  $ is an  arithmetic group and  is of
finite index  in the integral unitary  group $U(h)(\Z)$ where  $h$ is a
suitable Hermitian form defined over the rationals. \\

If $d\geq 3$ and $n\geq 2d$ is  such that $n+1$ is coprime to $d$, then
$\Gamma $  is an arithmetic  group. $\Gamma $  has finite index  in an
integral  unitary  group  of  the  form  $U(h)(O_d)$  where  $h$  is  a
non-degenerate  Hermitian form over  the  totally  real number  field
$K_d=\Q({\rm  cos}(\frac{2\pi  }{d}))$,  and  $O_d$  is  the  ring  of
integers in $K_d$.
\end{theorem}

Among these examples, only the case $d=2$ (treated in \cite{A'C}) has the
property that  the monodromy  group $\Gamma $  can be  conjugated into
$GL_n(\Z)$ (in fact it already  lies in $Sp_{2m}(\Z)$) with respect to
the natural representation given in \cite{Beu-Hec}). \\

As we  will see, Theorem \ref{beukers}  is an easy  consequence of the
fact that  the image  of the Burau  representation of the  braid group
$B_{n+1}$ on $n+1$  strands at $d$-th roots of  unity is an arithmetic
group in  the cases stated in the  theorem.  In case the  image of the
Burau  representation is  not arithmetic  (\cite{Del-Mos}, \cite{Mc}),
the image of  the monodromy representation defines a  {\it thin group}
in the sense of Sarnak (\cite{Sar}).\\

That the image of the Burau representation for $d=2$ is a finite index
subgroup of the  integral symplectic group is a  well known theorem of
A'Campo  (\cite{A'C}).  In  the other  cases, this is 
Theorem \ref{mainth}. \\

We now  sketch the relationship between Theorem  \ref{beukers} and the
Burau  representation.  Suppose   $D$  is  the  differential  operator
considered at the beginning of this section. Put $a_j=e^{2\pi i \alpha
_j}$  and write  $f(X)=\prod _{j=1}  ^n (X-a_j)$.   Consider  the ring
$\C[X]/(f(X))$.   This is  a $\C$-vector  space with  the  basis $1,X,
\cdots, X^{n-1}$. The operator defined by multiplication by $X$ on the
ring $\C[X]/(f(X))$ gives a matrix $A$ with respect to this basis, and
is  called the companion  matrix of  $f$.  Similarly,  let $B$  be the
companion matrix of  $g(X)=\prod _{j=1}^n (X-b_j)$, where $b_j=e^{2\pi
i \beta  _j}$. We  will assume henceforth  that $a_j\neq b_k$  for any
$j,k$; i.e. $f,g$ are coprime. \\

By  results  of Levelt  (\cite  {Beu-Hec})  there  exists a  basis  of
solutions $\{u\}=\{u_1, \cdots, u_n\}$ of the equation $Du=0$ on which
the monodromy action of $\pi _1(C)=\pi _1(C, \frac12)$ is described as
follows.   Let  $h_0,h_1,h_{\infty}$  be  small  loops  in  $C$  going
counterclockwise around $0,1,\infty$  exactly once. They generate $\pi
_1(C, \frac12)$ and  satisfy the relation $h_0h_1h_{\infty}=1$.  Under
the monodromy representation $M_{\alpha,  \beta}$, the matrix of $h_0$
is $A$, that of $h_{\infty}$ is  $B^{-1}$ (then the matrix of $h_1$ is
$A^{-1}B$). It  follows that $A^{-1}B$  is a {\it  complex reflection}
i.e.  the  space of  vectors fixed by  $A^{-1}B$ is a  codimension one
subspace. \\

Moreover,  suppose   $X,Y  \in  GL_n(\C)$   are  such  that   (1)  the
characteristic  polynomial  of  $X$  is $f$,  and  the  characteristic
polynomial  of $Y$  is  $g$, (2)  the  matrix $X^{-1}Y$  is a  complex
reflection.  We then  get  a  representation $M'$  of  $\pi _1(C)$  by
sending  $h_0$ to  $X$ and  $h_{\infty}$ to  $Y^{-1}$.  The  result of
Levelt is  that the representation  $M'$ is equivalent  to $M_{\alpha,
\beta}$ for some $\alpha, \beta$  such that $e^{2\pi i \alpha _j}$ for
varying $ j$  give all the roots of $f$ ($\beta  $ is chosen similarly
for $g$). \\

Now  consider  the   Burau  representation  $\rho  _n:  B_{n+1}\ra
GL_n(\Z[q,q^{-1}])$. The braid group $B_{n+1}$ is generated by the two
elements    $t_0=s_1s_2\cdots   s_n$   and    $t_1=s_n^{-1}$.    Write
$t_{\infty}=(s_1s_2\cdots     s_{n-1})^{-1}$.      We     then     get
$t_0t_1t_{\infty}=1$.  Therefore, we have a surjection from $\pi _1(C,
\frac12)\ra  B_{n+1}$  given  by  $h_0\mapsto t_0$  and  $h_{\infty  }
\mapsto  t_{\infty}$.    Composition  of  this  map   with  the  Burau
representation $\rho _n$ gives  a representation $r: \pi _1(C, \frac12)
\ra GL_n(\Z[q,q^{-1}])$.  Put $X=r (h_0)$  and $Y=r(h_{\infty}^{-1})$.
Then $X^{-1}Y=r(h_1)=\rho _n(t_1) =\rho _n(s_n^{-1})$, and the formula
for the Burau representation shows that the latter matrix is a complex
reflection.  Secondly, with respect to the standard basis $e_i$ of the
Burau representation, the element $h_0$ has the matrix form
\[X= \rho _n(q)(h_0)= 
\begin{pmatrix} 0 & 0 & 0 & \cdots & -q  \\ q & 0 & 0 & \cdots & -q \\
0 & q & 0 & \cdots & -q \\ \cdots & \cdots & \cdots & \cdots & -q \\ 0
& \cdots  & \cdots &  q & -q  \end{pmatrix} \] and  its characteristic
polynomial is of the form
\[Ch(t,X)= \prod  _{j=1}^n (t-q~e^{2\pi i j/(n+1)})=f(t),\] say. Similarly, if 
$Y=  \rho _n(q)(h_0^{-1})$  then  the  characteristic
polynomial is of the form
\[Ch (t,Y)=(t-1) \prod  _{k=1}^{n-1} (t-q~e^{2\pi i k/n})=g(t),\] say.
It  is clear that  the two  characteristic polynomials  do not  have a
common root. Therefore,  by Levelt's result, $r$ is  equivalent to the
monodromy   representation   $M_{\alpha,   \beta}$   of   a   suitable
hypergeometric  equation associated to  parameters $\alpha  _j$, where
$a_j=qe^{2\pi j/(n+1)}=e ^{2\pi \alpha  _j}$, and $b_j=e^{2\pi i \beta
_j}=qe^{2\pi  ij/n}~{\rm or}  ~1$.   Specialise $q$  to any  primitive
$d$-th root of unity  $e^{2\pi c/d}$. The resulting representation $r$
of  the  group  $\pi  (C,\frac12)$  is  equivalent  to  the  monodromy
representation associated  to the parameters  $\alpha , \beta$  in the
Theorem, and therefore, it has the same image (up to conjugacy) as the
Burau representation  $\rho _n (d)$.  Therefore,  the arithmeticity of
$M_{\alpha,\beta}$ follows from Theorem \ref{mainth}. \\

In Theorem \ref{beukers},  we have proved that a  very special case of
the representation $M_{\alpha, \beta}$ coincides with the monodromy of
the  Burau representation; therefore,  its arithmeticity  follows from
Theorem \ref{mainth} and from \cite{A'C}.

\section{Theorem \ref{cyclicmonodromy} and its proof} \label{cyclicmonodromysection}

Let $d\geq 2$  and $k_1,k_2, \cdots, k_{n+1}$ be  integers with $1\leq
k_i \leq d-1$. Let $a_1, \cdots, a_{n+1}$ be distinct complex numbers.
Consider  the  affine curve  $X_{a,k}=\{(x,y)\in  \C ^2\}$  given by  the
equation
\[y^d=(x-a_1)^{k_1}\cdots  (x-a_{n+1})^{k_{n+1}}.\]   $X_{a,k}$  is  a
compact Riemann surface $X_{a,k}^*$ with finitely many punctures.  The
space  $S$  of $a=(a_1,  \cdots,  a_{n+1})\in  C^{n+1}$ with  distinct
co-ordinates  has fundamental  group  isomorphic to  the ``pure  braid
group'' denoted  $P_{n+1}$. It  is the kernel  to the  map $B_{n+1}\ra
S_{n+1}$ (see the last paragraph of subsection \ref{artinaction}).  As
before, we have a  family $X \ra S$ with the fibre  over $a$ being the
compact  Riemann  surface  $X_{a,k}^*$   and  we  have  the  monodromy
representation    of    $P_{n+1}$     on    the    cohomology    group
$H^1(X_{a,k}^*,\Z)$.

\begin{theorem} \label{cyclicmonodromy} If  all the $k_i$ are co-prime
to $d$  and if  $n\geq 2d$, then  the image  $\Gamma $ of  the monodromy
representation is an arithmetic group.
\end{theorem}

We only sketch the proof (the proof is much more involved than that of
Theorem    \ref{fullbraidmonodromy}).     The    proof   of    Theorem
\ref{fullbraidmonodromy}   used   the    properties   of   the   Burau
representation, which were  established in Section \ref{burausection}.
The proof of Theorem  \ref{cyclicmonodromy} is quite similar, but uses
properties   of  the   reduced  {\it   Gassner}   representation  (see
\cite{Bir},  p.119   for  the   definition  of  the   reduced  Gassner
representation).  This is a representation
\[g_n(X):  P_{n+1}  \ra   GL_n(\Z[X_1^{\pm  1},  \cdots,  X_{n+1}^{\pm
1}]),\] of the  pure braid group $P_{n+1}$ on the  free module of rank
$n$ over the ring of Laurent polynomials with integral coefficients in
$n+1$ variables $X_1, \cdots,  X_{n+1}$. If $z_1, \cdots, z_{n+1}$ are
complex numbers, then  we get a specalisation $g_n(z)$  of the reduced
Gassner representation, called the {\it reduced Gassner representation
evaluated at} $z_1, \cdots, z_{n+1}$. \\

The properties  of Gassner  representation which we  will use  are the
following (for the second part of the Proposition, see \cite{Abd}).

\begin{proposition}  \label{Gassnerproposition}  The  reduced  Gassner
representation  is  has  a  non-degenerate  skew  Hermitian  form  $H$
preserved by $P_{n+1}$.  It  is absolutely irreducible.  The centre of
$B_{n+1}$ (which  lies in $P_{n+1}$  and is generated by  $\Delta ^2$)
acts  by scalars,  and $\Delta  ^2$ acts  by the  scalar $X_1X_2\cdots
X_{n+1}$ on the Gassner representation.\\

If we  specialise $X_i\mapsto z_i=q^{k_i}$ where $k_i$  are coprime to
$d$ and $q$ is a generator of the cyclic group $\Z/d\Z=q^{\Z}/q^{d\Z}$
, then  the reduced Gassner  representation evaluated at  these $d$-th
roots of  unity is irreducible unless  $z_1z_2\cdots z_{n+1}=1$. \\

If $z_1\cdots z_{n+1}=1$, then reduction  of the Hermitian form $H$ is
degenerate  and  its null  space  is  one  dimensional. Moreover,  the
quotient is irreducible.
\end{proposition}

In  the  Burau  case, we  used  the  fact  that if  $n\equiv  -1~({\rm
mod}~d)$,   then  the  Hermitian   form  is   degenerate  (Proposition
\ref{burauirred}) to  produce (many) unipotent  elements. We similarly
use the  last part  of Proposition \ref{Gassnerproposition}  to obtain
unipotent elements in the Gassner case. \\

The  analogue of  Theorem \ref{mainth}  is  the following.  As in  the
introduction, let $A_d=\Z[q,q^{-1}]/(\Phi  _d(q))$; it is isomorphic to
the  ring of  integers in  the  $d$-th cyclotomic  extension $E_d$  of
$\Q$.  Let $O_d$  denote  the ring  of  integers of  the totally  real
sub-field $K_d=\Q[q+q^{-1}]/(\Phi _d(q))$ of $E_d$.

\begin{theorem}     \label{gassnermainth}    Let    $g_n(d):P_{n+1}\ra
GL_n(A_d)$  denote  the reduced  Gassner  representation evaluated  at
$X_i=q_0 ^{k_i}$ where  $k_i$ are coprime to $d$,  and $q_0\in A_d$ is
the image of $q$. \\

If $n\geq 2d$, then the image of $g_n(d)$ is an arithmetic subgroup of
a suitable unitary group.
\end{theorem}

The proof of Theorem \ref{gassnermainth} is similar to that of Theorem
\ref{mainth}.   In the  case  of the  Burau  representation, after  we
constructed sufficiently many unipotent elements, Theorem \ref{mainth}
could  be proved  using the  fact  (see the  section on  the proof  of
Theorem  \ref{mainth}) that  the  $K$-rank of  the associated  unitary
group was $\geq 2$, provided $n\geq 2d$.  The argument was as follows.
Write $n+1= d+m+d$,  where $m\geq 1$.  Then the span  of the first $d$
basis elements  $e_1, \cdots, e_d$ of the Burau representation  
contains an isotropic vector  $v$ ,
and similarly  the span of the  last $d$ basis  elements $e_n, \cdots,
e_{n-d+1}$ contains an  isotropic vector $v'$. If $m\geq  1$, then the
two  sets  of basis  elements  are  orthogonal  and hence  $v,v'$  are
orthogonal for the hermitian form $h_n$  and hence the $K$ rank of the
unitary group is at least 2. \\

We argue similarly  in the case of the  Gassner representation. As was
remarked at  the end  of Proposition \ref{Gassnerproposition},  we can
get unipotent  elements if there are  subsets $X$ of  the indexing set
$1,2, \cdots,n$ such that $\prod _{i\in X} z_i=1$.  Put $z_i=q^{k_i}$;
an  argument using the  pigeon-hole principle  implies that  if $n\geq
2d$,  then   there  are  two   disjoint  subsets  $X,Y$  of   the  set
$\{1,2,\cdots,  n\}$  such that  $\prod
_{i\in  X} z_i=1$  and  $\prod  _{j\in Y}z_j=1$  (In  the Burau  case,
$X=\{1,2,\cdots, d\}$ and  $Y=\{n,n-1, \cdots, n-d+1\}$ will suffice).
By the  third part of Proposition  \ref{Gassnerproposition} it follows
that  the span of  $e_i$ for  $i\in X$  contains an  isotropic vector
$v_X$.   Similarly,  the span  of  $e_j$  for  $j\in Y$  contains  an
isotropic vector $v_Y$.  The Hermitian  form preserved by the image of
the  Gassner representation  is such  that if  $X,Y$ are  disjoint and
their  union  is  a  proper  subset  of  $\{1,2,\cdots,  n\}$,  then
$v_X,v_Y$  are orthogonal, and  hence the  $K$-rank of  the associated
unitary    group   is   at    least   two.     Moreover,   Proposition
\ref{Gassnerproposition} applied to  the set $X$ (in place  of the set
$1,2,\cdots, n$) implies  that we have many unipotent  elements in the
image  of  the  Gassner   representation.   By  appealing  to  Theorem
\ref{bamise}, we then deduce Theorem \ref{gassnermainth}. \\

The  proof of  Theorem \ref{cyclicmonodromy}  is deduced  from Theorem
\ref{gassnermainth}, by relating  the monodromy representation, to the
Gassner representation. The proof of this relationship is very similar
to the proof in Section \ref{monodromyandburau} relating monodromy and
the Burau  representation (but is much  more involved and  we omit the
details).

\end{document}